\documentclass[11pt,a4paper,oneside]{amsart}

\usepackage[a4paper,margin=3cm]{geometry}
\usepackage{amsmath,amssymb,amsthm,mathrsfs,mathtools}
\usepackage{booktabs, tabu}
\usepackage{pdflscape}
\usepackage{esint}

\usepackage{hyperref}
\usepackage{cleveref}

\usepackage[dvipsnames]{xcolor}	
\usepackage[textsize=footnotesize]{todonotes}
\presetkeys%
    {todonotes}%
    {color=Bittersweet}{}%
\theoremstyle{plain}                 
\newtheorem{theorem}{Theorem}[section]
\newtheorem{conjecture}[theorem]{Conjecture}
\newtheorem{proposition}[theorem]{Proposition} 
\newtheorem{corollary}[theorem]{Corollary}     
\newtheorem{lemma}[theorem]{Lemma}        
\theoremstyle{definition}           
\newtheorem{definition}{Definition}    
 
\theoremstyle{remark}       
\newtheorem{remark}{Remark}

\def\ch{\mathrm{ch}}

\def\oM{\overline{\mathcal{M}}}

\def\CP1{\mathbb{C}\mathrm{P}^1}

\def\zz{{\scriptstyle\mathcal{Z}}}

\newcommand{\Klog}{\omega_{\rm log}}
\newcommand{\cM}{{\mathcal{M}}}
\newcommand{\cL}{{\mathcal{L}}}
\newcommand{\SSS}{\mathcal{S}}
\newcommand{\sta}{\mathcal}

\newcommand{\MMMbar}{\overline{\sta M}}

\newcommand{\cor}[1]{{\langle} \, #1 \, {\rangle}}

\newcommand{\II}[2]{\int_{\overline{\mathcal{M}}_{#1, #2}} \! }

\DeclarePairedDelimiter\floor{\lfloor}{\rfloor}

\newcommand{\bea}{\begin{eqnarray}}
\newcommand{\eea}{\end{eqnarray}}
\newcommand{\ben}{\begin{eqnarray*}}
\newcommand{\een}{\end{eqnarray*}}
\newcommand{\be}{\begin{equation}}
\newcommand{\ee}{\end{equation}}

\newcommand{\highlightr}[1]{%
  \colorbox{red!50}{$\displaystyle#1$}}
\newcommand{\highlightg}[1]{%
  \colorbox{green!50}{$\displaystyle#1$}}

\graphicspath{{./figures/}} 
\usepackage[style=alphabetic,maxalphanames=5,giveninits,sorting=nyt,%
maxbibnames=99,backend=bibtex]{biblatex} 
\begin{document}

\title[A natural basis for intersection numbers]{A natural basis for intersection numbers}
\author[B.~Eynard]{B.~Eynard}
\address{B.~E.: Inst. de Physique Th\'{e}orique (IPhT),
Commissariat \`{a}  l'Energie atomique,
Orme des Merisiers Bat 774,
91191 Gif-sur-Yvette, 
Paris, France \&
Inst. Hautes \'{E}tudes Scientifiques (IHES),
35 Route de Chartres, 
91440 Bures-sur-Yvette, 
Paris, France }
\email{bertrand.eynard@ipht.fr}
\author[D.~Lewa\'{n}ski]{D.~Lewa\'{n}ski}
\address{D.~L.: 
Inst. de Physique Th\'{e}orique (IPhT),
Commissariat \`{a}  l'Energie atomique,
Orme des Merisiers Bat 774,
91191 Gif-sur-Yvette, 
Paris, France \&
Inst. Hautes \'{E}tudes Scientifiques (IHES),
35 Route de Chartres, 
91440 Bures-sur-Yvette, 
Paris, France 
}
\email{danilo.lewanski@ipht.fr}
\begin{abstract} We advertise elementary symmetric polynomials $e_i$ as the natural basis for generating series $A_{g,n}$ of intersection numbers of $\psi$-classes on the moduli space of stable curves of genus $g$ with $n$ marked points. Closed formulae for $A_{g,n}$ are known for genera $0$ and $1$ \textemdash\; this approach provides formulae for $g=2,3,4$, together with an algorithm to compute the formula for any $g$. 

The claimed naturality of the $e_i$ basis relies in the unexpected vanishing of some coefficients with a clear pattern: we conjecture that $A_{g,n}$ can have at most $g$ factors $e_i$, with $i>1$, in its expansion. This observation promotes a paradigm for the intersection numbers with more general cohomology classes. As an application of the conjecture, we find new integral representations of $A_{g,n}$, which recover expressions for the Weil-Petersson volumes in terms of Bessel functions. 
\end{abstract}
\maketitle
\tableofcontents

\noindent


\section{Introduction}


Over the last 30 years, research in mathematics and in theoretical physics have unveiled a deep interaction between the following fields:
\begin{enumerate}
\item models in theoretical physics (mainly arising from string theory and random matrix models), 
\item the algebraic geometry of moduli spaces of curves, 
\item the mathematical physics of integrable systems and integrable hierarchies.
\end{enumerate}

\noindent
Let us present the core idea of this interaction. Consider correlators 
$$ \langle \eta_{d_1} \cdots \eta_{d_n} \rangle_g^M $$
of some theoretical physics model $M$ as above depending on a genus parameter $g$ and on a certain number $n$ of insertions. On the other hand, consider a partition function $\mathcal{Z}(\hbar, \textbf{t})$ depending on a formal parameter $\hbar$ and infinitely many parameters $t_i$ in such a way that it makes sense to define its correlators
\begin{equation}
\langle \sigma_{d_1} \cdots \sigma_{d_n} \rangle_g^{\mathcal{Z}} := \left[\frac{\hbar^{2g - 2 + n}}{n!} t_{d_1} \cdots t_{d_n}\right]. \log \mathcal{Z}(\hbar, \textbf{t})
\end{equation}
where $[x^d].f(x)$ is the formal operator which extracts the coefficient of $x^d$ in $f(x)$.
Assume moreover that $\mathcal{Z}$ is the tau-function of some integrable hierarchy (e.g. KdV, KP, $2$D Toda, BKP, $\dots$), that is, it satisfies a determined infinite list of PDEs in the infinite set of times $t_i$. Finally, define the correlators 
\begin{equation}
 \langle \tau_{d_1} \cdots \tau_{d_n} \rangle_g^{\Omega} := \int_{\overline{\mathcal{M}}_{g,n}} \Omega_{g,n} \psi_1^{d_1} \cdots \psi_n^{d_n}, \qquad \qquad \Omega_{g,n} \in H^{*}(\overline{\mathcal{M}}_{g,n}),
\end{equation}
where $\overline{\mathcal M}_{g,n}$ is the $3g - 3 + n$ complex dimensional
moduli space of stable curves \((C; p_1, \ldots, p_n)\) of genus $g$, with $n$
distinct labeled marked points and the cohomology classes $\psi_i :=
c_1(\mathcal{L}_i)$ are the first Chern classes of the line bundle
$\mathcal{L}_i$ cotangent at the $i$-th marked point. Here $\Omega$ is a whole
collection of classes $\Omega_{g,n}$ for every $(g,n)$ satisfying $2g - 2 + n >
0$. In practice, these collections $\Omega$ form Cohomological Field Theories
(CohFTs), which means that several compatibility conditions between different
$\Omega_{g_i, n_i}$ in the same collection hold.

The interaction we are talking about can be thought as triples $(M,\mathcal{Z},\Omega)$ such that
\begin{equation}
 \langle \eta_{d_1} \cdots \eta_{d_n} \rangle_g^{M}  \sim  \langle \sigma_{d_1} \cdots \sigma_{d_n} \rangle_g^{\mathcal{Z}}  \sim  \langle \tau_{d_1} \cdots \tau_{d_n} \rangle_g^{\Omega},
\end{equation}
where the symbol $\sim$ in principle means proportionality up to the multiplication of some combinatorial prefactor which is $d_i$-dependent, but in some cases is a straightforward equality.

Let us make an example of this interaction. In fact, this example is the simplest possible instance, historically arose first,  and in a way can be thought as the generic-local model for all other examples of such interaction.  In 1991 two different approaches to $2$-dimensional quantum gravity, used as an easier model to understand the theory in higher dimension, existed. Both had to deal with the problematic integration over infinite dimensional spaces. The first was based on triangulation of surfaces, the second on the integration over the space of conformal metrics, which translates to the computation of correlators 
\begin{equation}
 \langle \tau_{d_1} \cdots \tau_{d_n} \rangle_g := \int_{\overline{\mathcal{M}}_{g,n}} \psi_1^{d_1} \cdots \psi_n^{d_n},
\end{equation}
that is, $\Omega_{g,n} = 1$ is the fundamental class of the moduli space, for every $(g,n)$. Witten conjectured these two approaches were the same. More precisely, as the first approach was known to obey the KdV integrable hierarchy, the initial conditions of both approaches were known to coincide, and both approaches were known to satisfy the so-called string equation, Witten conjectured that the second approach obey (and therefore is completely determined by) the KdV integrable hierarchy.
\begin{theorem}[Witten conjecture \cite{W91}, Kontsevich theorem \cite{K92}] \label{thm:WK}  The generating function
\begin{equation}\label{gen}
U=\frac{\partial^2F}{\partial t_0^2}, \qquad \qquad 
F(t_0, t_1, \ldots)= \sum_{g=0}^{\infty} \sum_{d_0, d_1, \dots = 0}^{\infty}
\langle \tau_{0}^{d_0} \tau_{1}^{d_1} \cdots \rangle_{g}
 \prod_{i=0} \frac{t_i^{d_i}}{d_i!}
\end{equation}
 satisfies the classical KdV equation
\begin{equation} \label{KdV}
\frac{\partial U}{\partial t_1}=U\frac{\partial U}{\partial t_0}+\frac{1}{12}\frac{\partial^3 U}{\partial t_0^3}.
\end{equation}
\end{theorem}

\noindent
One can think of this instance, in terms of the interaction presented above, as of the triple
$$
(M = 2D\text{ quantum gravity}, \mathcal{Z}  = \text{ string solution of the KdV hierarchy}, \Omega_{g,n} = 1).
$$
This interaction is just the tip of an entire iceberg of such triples: many are discovered, and yet the depth of the iceberg is still far from being assessed.

What is known about the explicit generating series of $ \langle \tau_{d_1} \cdots \tau_{d_n} \rangle_g$? The goal of this paper is two-fold: on the one hand it provides new closed formulae, on the other it notices (and it conjectures) an unexpected vanishing of the coefficients with a precise pattern, providing the proof for several cases.

\subsection{Low genus cases: $g=0$ and $g=1$}
Let us start by the low genus cases. The simplest equations these correlators satisfy are the \textit{string} and \textit{dilaton} equations:
\begin{align}
	& \cor{ \tau_{d_{1}} \ldots \tau_{d_{n}} \tau_{0} }_{g} =
	\sum_{i=1}^{n} 
	\cor{ \tau_{d_{1}}\ldots \tau_{d_{i}-1}\ldots \tau_{d_{n} }}_{g},
	\\
	& \cor{ \tau_{d_{1}} \ldots \tau_{d_{n}} \tau_{1} }_{g} =
	(2g-2+n)
	\cor{ \tau_{d_{1}} \ldots \tau_{d_{n}}}_{g}.
\end{align}

In genus zero, as the dimension of $\overline{\mathcal{M}}_{0,n}$ is $n-3$ and $\tau_i$ corresponds to cohomological degree $i$, insertions $\tau_0$ must appear, so it is customary to exploit string equation to prove
\begin{equation}\label{eq:g0int}
\II 0n \psi_1^{d_1} \cdots \psi_n^{d_n} = \binom{n-3}{d_1, \dots, d_n}.
\end{equation}
\begin{definition}
Define for $\textbf{x} = (x_1, \dots, x_n)$ the amplitudes $A_{g,n}$ as
\begin{align}
	\label{eq:Agndef}
	A_{g,n}(\textbf{x}) &=  \II g n 
	\frac{1}{\prod_{i=1}(1 - x_i \psi_i)} 
\\&= \sum_{d_1, \dots, d_n = 0} 
\cor{ \tau_{d_{1}} \ldots \tau_{d_{n}}}_{g} \; x_1^{d_1} \cdots x_n^{d_n}, 
\label{eq:Agnmonom}
\end{align}
and \footnote{summations without upper bound in this paper refer to all indices being summed up to, in principle, infinity. Although in some cases, like in this one, the sum is finite for dimension reasons.} the normalized amplitudes as
\begin{equation}\label{eq:Agndefnorm}
\tilde{A}_{g,n}  = {24^g g!} \, A_{g,n}.
\end{equation}
\end{definition}
\noindent
Equation \eqref{eq:g0int} translates into $A_{0,n}(\textbf{x}) = (x_1 + \dots + x_n)^{n-3}$, and rewriting it in terms of symmetric elementary polynomials $e_i$ gives
\begin{equation}\label{eq:A0n}
A_{0,n} = \tilde{A}_{0,n} = e_1^{n-3}.
\end{equation}
Notice that only factors of the type $e_1$ appear in genus zero, and that there is no obvious reason a priori why factors $e_j$, $j>1$, wouldn't show up.
The formula in genus one is also known (see e.g. \cite{LiuXu_npoint} and references within), and it reads
\begin{equation*}
\int_{\overline{\mathcal{M}}_{1,n}} \psi_1^{d_1} \cdots \psi_n^{d_n} 
=
 \frac{1}{24} \biggl( \binom{n}{d_1, \ldots, d_n} - \sum_{\substack{b_1, \ldots, b_n \\ b_i \in \{0,1\}}} \binom{n - (b_1 + \cdots + b_n)}{d_1 - b_1, \ldots, d_n - b_n} (b_1 + \cdots +b_n - 2)! 
 \biggr), 
\end{equation*}
where the convention that negative factorials vanish is used. In terms of the normalized amplitudes it becomes
\begin{equation}
\tilde{A}_{1,n} = e_1^{n} - \sum_{k=2}^n (k-2)! e_k \; e_1^{n-k}.
\label{eq:A1n}
\end{equation}
Again, the factors $e_j$, $j>1$, are strangely scarce in the expansion above: the formula holds for every $n$, therefore the dimension of the moduli space could be as big as desired.


\subsection{Results of the paper and main conjecture}
For $2g - 2 + n > 0$, the amplitudes $A_{g,n}$ are manifestly symmetric
polynomials of homogeneous degree $3g - 3 + n$. They can therefore be expressed
in the basis of elementary symmetric polynomials $e_{\Lambda}
= e_{\Lambda_1} \cdots e_{\Lambda_{\ell(\Lambda)}}
$ with partitions \( \Lambda \) of size \( |\Lambda|  = 3g-3+n \):
\begin{equation}
	A_{g,n}(\textbf{x}) 
	=
	\sum_{|\Lambda| = 3g-3+n}
	\!\!\!\! D_{g,n}(\Lambda) e_{\Lambda}(\textbf{x})  \label{eq:AgnInDgn}
\end{equation}
for some coefficients $D_{g,n}(\Lambda)$. Observe that any $\Lambda_i > n $ implies the vanishing of the whole summand, hence we can consider partitions with parts bounded by $n$.

%

Let us first distinguish without loss of generality the elements $\Lambda_i = 1$ in the sum above: 
\begin{equation}
	A_{g,n}(\textbf{x}) 
	=
	\sum_{\substack{
	|\lambda| \leq 3g-3+n 
	\\
	 2 \leq \lambda_{i} \leq n 
	 }} 
	\!\!\!\! C_{g,n}(\lambda) \; e_{\lambda} \; e_1^{3g - 3 + n - |\lambda|}
	 \label{eq:AgnInCgn}
\end{equation}
where we can set without ambiguity 
\begin{equation}
	\label{eq:defCgn}
C_{g,n}(\lambda) \coloneqq D_{g,n}(\lambda \sqcup
(1)^{3g-3+n-|\lambda|}),
\end{equation}
 with \( \lambda \sqcup (1) \coloneqq (\lambda_{1} ,\ldots, \lambda_{n} , 1)\).
We first show in Section \ref{sec:virasoro} (see corollary \ref{cor:33} and equation \ref{eq:string1}) the following property.

\begin{proposition}
The coefficients $C_{g,n}(\lambda)$ are independent of $n$.
\end{proposition}


We can therefore drop the corresponding index. Then, computing the
coefficients \( C_{g}(\lambda) \) we observed that all coefficients
with \( \ell(\lambda) > g \) happen to be vanishing:
\begin{conjecture}[Main conjecture]
	\label{conj:main} For $g \geq 0, n\geq 1$ and $2g - 2 + n>0$ we have
	\begin{equation}
		C_{g}(\lambda)= 0, \quad \text{ for } \quad  \ell(\lambda) >g.
	\end{equation}
	Equivalently,
\begin{equation}
	A_{g,n}(\textbf{x}) 
	=
	\sum_{\substack{|\lambda| \leq 3g-3+n \\ \lambda_{i} \geq 2 \\ l(\lambda) \leq g}}
	C_{g}(\lambda) \; e_{\lambda}  \; e_{1} ^{3g-3+n-|\lambda|}.   \label{eq:conj}
\end{equation}
\end{conjecture}
We prove the conjecture up to $g \leq 7$, for any $n$, and up to $n \leq 3$, for any $g$. We then apply the Virasoro constraints to the $A_{g,n}$ obtaining a recursion for
the coefficients $C_g(\lambda)$ in Section \ref{sec:virasoro}.  It is an
important point that the recursion does not rely on the conjecture; on the
opposite,  it allows to verify the conjecture up to $g \leq 7$,  by checking
that all expected vanishings do hold.  In addition, the recursion can be
employed to obtain explicit formulae for the amplitudes $A_{g,n}$ for
higher genera.  As an illustrative example, we give in the following the two
new formulae for $g=2$ and $g=3$. The formula for $g=4$ takes several pages and
is given in Appendix \ref{sec:Cgappendix}.

\begin{remark}
Let us remark that these formulae at the practical level allow for fast computations for large $n$, in genus $2,3$ and $4$, compared to the existing softwares exploiting Virasoro constraints.  One could argue that our methods simply employ Virasoro constraints applied to the basis of the $e_i$.  However, by polynomial considerations, we only need to compute up to $n = 6g - 3$ (see remark \ref{rmk:bound}) in order to fix unambiguously all needed coefficients \textemdash \; after that the intersection numbers can be computed for any given $n$.
\end{remark}
  
 \begin{proposition} The genus two formula reads:
\begin{equation} 
	\begin{aligned} 
	\tilde{A}_{2,n} &= e_1^{n+3} 
	- 2 e_2 e_1^{n+1} - \frac{18}{5}
	e_3 e_1^n - \sum_{k=4}^{n+3} 
	\frac{(k^{3} + 21 k^{2} - 70 k + 96)(k-3)!}{30}
	e_k e_1^{n +3 - k} \\ &+ \frac{9}{5} e_2^2
	e_1^{n-1} + \frac{18}{5} e_3 e_2 e_1^{n-1} + \sum_{k=4}^{n+1}
	\frac{\left(k + 16\right)\left(k - 1\right)!}{10} 
	 e_k e_2 e_1^{n + 1 - k} \\
							 &-\sum_{k=4}^n
						 \frac{k!}{10}  \,e_k e_3
					 e_1^{n-k}.  
\end{aligned}
\label{eq:A2n}
\end{equation}
The convention that \( e_{1}^{m} = 0 \) if \( m < 0 \) is used.
\end{proposition}

\begin{proposition}
The genus three formula reads:
\begin{equation} \label{eq:A3n}
	\begin{aligned}
\tilde{A}_{3,n}  &= \sum_{k=7}^{n + 4} \tfrac{\left(17 k^{4} + 814 k^{3} + 9391 k^{2} - 12142 k + 53904\right) \left(k - 2\right)! }{8400}{e}_{1}^{- k + n + 4} {e}_{2} {e}_{k} 
\\&- \sum_{k=7}^{n + 3} \tfrac{\left(2 k^{3} + 39 k^{2} - 1523 k - 480\right) \left(k - 1\right)! }{2100}{e}_{1}^{- k + n + 3} {e}_{3} {e}_{k} 
\\&- \sum_{k=5}^{n + 2} \tfrac{\left(5 k^{2} + 199 k + 2282\right) k! }{1400}
{e}_{1}^{- k + n + 2} {e}_{2}^{2} {e}_{k}
- \sum_{k=7}^{n + 2} \tfrac{\left(3 k^{2} + 79 k + 596\right) k! }{700}
{e}_{1}^{- k + n + 2} {e}_{4} {e}_{k} 
\\&+ \sum_{k=7}^{n + 1} \tfrac{\left(k + 37\right) \left(k + 1\right)! }{350} 
{e}_{1}^{- k + n + 1} {e}_{5} {e}_{k}
+ \sum_{k=5}^{n + 1} \tfrac{\left(5 k + 102\right) \left(k + 1\right)! }{700}
{e}_{1}^{- k + n + 1} {e}_{2} {e}_{3} {e}_{k}
\\&+ \sum_{k=4}^{n} \tfrac{\left(k + 2\right)! }{175}
{e}_{1}^{- k + n} {e}_{2} {e}_{4} {e}_{k} 
- \sum_{k=3}^{n}  \tfrac{9 \left(k + 2\right)! }{1400} {e}_{1}^{- k + n} {e}_{3}^{2} {e}_{k} 
- \sum_{k=6}^{n}  \tfrac{\left(k + 2\right)! }{350}{e}_{1}^{- k + n} {e}_{6} {e}_{k} \\
&- \sum_{k=7}^{n + 6} \tfrac{\left(17 k^{6} + 885 k^{5} + 9347 k^{4} - 83577 k^{3} + 338972 k^{2} - 912492 k + 970272\right) \left(k - 4\right)! }{50400} {e}_{1}^{- k + n + 6} {e}_{k} + a_{3,n}
	\end{aligned}
\end{equation}
where $a_{3,n}$ is the collection of initial terms which do not naturally arise in $k$-families:
\begin{equation}
	\begin{aligned}
	&a_{3,n} \nonumber  = \displaystyle - \tfrac{27 }{7}{e}_{1}^{n} {e}_{2}^{3} + \tfrac{1692 }{35}{e}_{1}^{n} {e}_{2} {e}_{4} + \tfrac{153 }{35}{e}_{1}^{n} {e}_{3}^{2} - \tfrac{1872 }{5} {e}_{1}^{n} {e}_{6}+ \tfrac{3024}{5} {e}_{1}^{n - 5} {e}_{5} {e}_{6} 
- \tfrac{8496 }{7}{e}_{1}^{n - 4} {e}_{4} {e}_{6} \\
		&+ \tfrac{432 }{5}{e}_{1}^{n - 4} {e}_{5}^{2} 
		+ \tfrac{108 }{5}{e}_{1}^{n - 3} {e}_{2} {e}_{3} {e}_{4} + \tfrac{15552 }{35}{e}_{1}^{n - 3} {e}_{3} {e}_{6} - \tfrac{5904 }{35}{e}_{1}^{n - 3} {e}_{4} {e}_{5} - 54 {e}_{1}^{n - 2} {e}_{2}^{2} {e}_{4} \\
		&+ \tfrac{27 }{7}{e}_{1}^{n - 2} {e}_{2} {e}_{3}^{2} 
		+ \tfrac{51696 }{35}{e}_{1}^{n - 2} {e}_{2} {e}_{6} + \tfrac{2844 }{35}{e}_{1}^{n - 2} {e}_{3} {e}_{5} - \tfrac{1152 }{35}{e}_{1}^{n - 2} {e}_{4}^{2} - \tfrac{81 }{7}{e}_{1}^{n - 1} {e}_{2}^{2} {e}_{3} \\
		&+ \tfrac{8532 }{35}{e}_{1}^{n - 1} {e}_{2} {e}_{5} + \tfrac{324 }{35} {e}_{1}^{n - 1} {e}_{3} {e}_{4}
+ \tfrac{594 }{35}{e}_{1}^{n + 1} {e}_{2} {e}_{3} - \tfrac{2286 }{35}{e}_{1}^{n + 1} {e}_{5} + \tfrac{27 }{5}{e}_{1}^{n + 2} {e}_{2}^{2} 
	      \\&- \tfrac{594 }{35}{e}_{1}^{n + 2} {e}_{4} - \tfrac{39 }{5}{e}_{1}^{n + 3} {e}_{3} - 3 {e}_{1}^{n + 4} {e}_{2} + {e}_{1}^{n + 6}. 
	\end{aligned}
\end{equation}
\end{proposition}

\begin{proposition}
The genus four formula is given in Appendix \ref{sec:Cgappendix}.
\end{proposition}

\subsection{Applications of the conjecture} 
As an application of conjecture \ref{conj:main},  
we obtain in section \ref{sec:gdimintegral} formulae for the amplitudes $A_{g,n}$ as 
$g+1$ dimensional integrals. In the following we rewrite the result in the residue form:  

\begin{proposition} Conjecture \ref{conj:main} implies that
\begin{equation}
	A_{g,n}(\mathbf{x})
=
[ u^{0} v_{1}^{0} \ldots v_{g}^{0}] 
\,
\frac{B^{-}_{g,n}(u,\mathbf{v})}{u^{d_{g,n}} }
\prod_{i,m} e^{u x_{i}} 
(1 + v_{m} x_{i} ),
\end{equation}
or\footnote{Compared to the expression in equation \ref{eq:conj}, this formulation has the advantage of being packaged in generating series and suggests a reformulation in terms of residues.}, equivalently,
		\begin{equation}
	\cor{
		\tau_{d_{1}} \ldots \tau_{d_{n}} 
}_{g}
=
[ u^{0} v_{1}^{0} \ldots v_{g}^{0}] 
\,
B^{-}_{g,n}(u,\mathbf{v})   
		\prod_{i}
\sum_{r = 0}^{ d_{i}}
			\frac{e_{r}(\mathbf{v})}{u^{r}}
			\frac{1}{(d_{i} - r)!},
		\end{equation}
\end{proposition}
\noindent
where \( d_{g,n} = 3g-3+n \), the operator $[x^k]$ extracts the coefficient of $x^k$ in the expression to which it is applied, and \( B^{-}_{g,n}(u, \mathbf{v}) \) is a polynomial
in \( u \) and in the \( 1/v_{i} \) given by the
coefficients \( C_{g}(\lambda) \) in the elementary symmetric 
basis decomposition of \( A_{g,n} \):
\begin{equation}
	B^{-}_{g,n}(\xi,v_1,\dots,v_g)=
\sum_{\substack{ |\lambda| \leq d_{g,n} \\ \lambda_{i} \geq 2\\ l(\lambda) \leq g }} 
(d_{g,n}-|\lambda|)! \ \frac{C_g(\lambda)}{g!} 
\tilde{m}_\lambda(\mathbf{1/v}) \xi^{|\lambda|}.
\end{equation}
Here \( \tilde{m}_{\lambda} \) are the augmented monomial symmetric
polynomials. 
We employ the same argument to extend the integral representation to expressions of correlators of the
form
\begin{equation}
	\cor{e^{ \sum_{k \geq 2} t_{k-1} \tau_{k}} 
	\prod_{i = 1}^{n} \tau_{d_{i}}}_{g}.
\end{equation}
In particular, we can apply our results to Weil-Petersson polynomials 
$$
V_{g,n}(\mathbf{L}) := \int_{\overline{\mathcal{M}}_{g,n}} \exp\left(2\pi^2 \kappa_1 + \sum_{i=1}^n \frac{L_i^2 \psi_i}{2}\right),
$$ 
finding the following expression.
\begin{theorem} Conjecture \ref{conj:main} implies that
\begin{equation}
		\begin{aligned}
			V_{g,n}(\mathbf{L}) 
			&=
	\sum_{\alpha_{1} ,\ldots, \alpha_{n}} 
	\cor{ e^{2 \pi^2 \kappa_{1}} \prod_{i = 1}^{n} \tau_{d_{i}}}_{g} 
	\prod_{i} \frac{L^{2 d_{i}}}{2^{d_{i}} d_{i}!} 
	\\
			&= 
	\, \underset{\mathbf{v} = 0}{\operatorname{Res}} \,\prod_{j=1}^{g}  
\frac{dv_j}{v_j} 
\, \underset{u = 0}{\operatorname{Res}} \,
du
\frac{B^{-}_{g,n}(F(u,\mathbf{v}),\mathbf{v})   
}{F(u,\mathbf{v})^{3g-2+n} }
\prod_{i = 1}^{n} 
(
\sum_{k} e_{k}(\mathbf{v}) 
\left( \frac{L_{i}}{\sqrt{2u}} \right)^{k} I_{k}(L_{i} \sqrt{2u})
)
\\
			&=
			[u^{-1} v_{1}^{0} \ldots v_{g}^{0}]
\frac{B^{-}_{g,n}(F(u,\mathbf{v}),\mathbf{v})   
}{F(u,\mathbf{v})^{3g-2+n} }
\prod_{i = 1}^{n} 
(
\sum_{k} e_{k}(\mathbf{v}) 
\left( \frac{L_{i}}{\sqrt{2u}} \right)^{k} I_{k}(L_{i} \sqrt{2u})
),
\end{aligned}
\end{equation}
where $I_k$, $J_k$ are standard notations for Bessel functions, and
\begin{equation}
	F(u,\mathbf{v}) = \sum_{k \geq 0}\Bigg[
	\left(-\frac{\pi \sqrt{2}}{\sqrt{u}}\right)^{k-1}\!\!\!\!\!
	J_{k-1}(2 \pi \sqrt{2u}) - \delta_{k,1} \Bigg] e_{k}(\mathbf{v}).
\end{equation}
\end{theorem}

\subsection*{Structure of the paper} In Section \ref{sec:BuryakOkounkov}  we summarise the state-of-the-art for generating series of intersection numbers. In Section \ref{sec:virasoro} we exploit the information carried by Virasoro constraints to refine our conjecture, derive a recursion of the coefficients of the generating series, and give some restatements of the main conjecture.  Section \ref{sec:n1} is dedicated to the proof of the main conjecture for $n=1,2,3$.  In Section \ref{sec:omega} we analyse the ELSV formula for one-part Hurwitz numbers and give a different restatement of our conjecture in terms of the $\Omega$-CohFT.  Section \ref{sec:gdimintegral} contains applications of the main conjecture as new formulae for $A_{g,n}$ and Weil-Petersson polynomials as $g$-dimensional integral representations.  Section \ref{sec:examples} contains a few examples of non-trivial cohomological field theories $\Omega_{g,n}$ which show similar behaviour to $\Omega_{g,n} = 1$ when their amplitudes are expanded in elementary symmetric polynomials.  Finally, we conclude with two appendices with numerics: the first tests the simplifications given both by our recursion and by the main conjecture, the second provides the formula of $A_{g,n}$ for $g=4$.

\subsection*{Acknowledgements}

B.~E. and D.~L. are supported by the \textit{Institut de Physique Théorique Paris} (IPhT) and the \textit{Institut de Hautes Études Scientifiques}. This work is partly a result of the ERC-SyG project, Recursive and Exact New Quantum Theory (ReNewQuantum) which received funding from the European Research Council (ERC) under the European Union's Horizon 2020 research and innovation programme under grant agreement No 810573. D.~L. is moreover supported by the INdAM group GNSAGA for scientific visits. D.~L. would like to thank Johannes Schmitt for the many conversations over the Sage package \textup{admcycles} \cite{admcycles} which was used to compute intersections over the moduli space of curves in Section  \ref{sec:examples}. Most importantly, this work would not have been possible without the fundamental contribution of Adrien Ooms, who used to be coauthor of this paper, before he asked the removal of his name from all his scientific projects at once.


\section{Known formulas for the $n$-point functions} \label{sec:BuryakOkounkov} 

The $n$-point function 
$$
F_n(\textbf{x}) := \sum_{g=0}A_{g,n}(\textbf{x}) = \int_{\overline{\mathcal{M}}_{g,n}} \frac{1}{\prod_{i=1}^n (1 - x_i \psi_i)}
$$
 is an alternative way to encode all
information of intersection numbers of $\psi$ classes.  In the expression above the genus is determined by the cohomological degree $D$ taken as $g = \frac{D + 3 - n}{3}$,  and the expression vanishes whenever the fraction is not an integer.  In the following we summarise the state-of-the-art about the $F_n$.

Okounkov \cite{Ok} obtains an analytic expression of the $n$-point functions
in terms of $n$-dimensional error-function-type integrals, based on
his work of random permutations. Buryak (see \cite{AIS_OkBur}) obtains another integral representation of the $F_n$ formula from the semi-infinite wedge formalism. Br\'ezin and Hikami \cite{BH} apply
correlation functions of GUE ensemble to find explicit formulae of $n$-point functions. Liu and Xu \cite{LiuXu_npoint}
 exploit the information carried by the Virasoro constraints to derive a recursive formula for the $F_n$.

\subsection{Buryak formula} Define the function $P_n(a_1,\dots,a_n;x_1,\dots,x_n)$ by $P_1(a_1;x_1) \coloneqq \frac{1}{x_1}$ for $n=1$, and for $n\geq 2$ by
\begin{align}\label{eq:DefinitionP}
P_n(\textbf{a}; \textbf{x}) \coloneqq \sum_{\substack{\tau \in \mathfrak{S}_n\\ \tau(1) = 1 }}
\frac{ 
\prod\limits_{j=2}^{n-1} x_{\tau(j)} 
\prod\limits_{j=1}^{n-1}
\zeta \left( 
	\left(\sum\limits_{k=1}^{j} a_{\tau(k)}\right)  x_{\tau(j+1)} -  a_{\tau(j+1)} \left( \sum\limits_{k=1}^{j} x_{\tau(k)} \right)
 \right) 
} {
	\prod\limits_{j=1}^{n-1} \left(
	a_{\tau(j)}x_{\tau(j+1)} - a_{\tau(j+1)} x_{\tau(j)} 
	\right)
},
\end{align}
where $\varsigma(z) = 2\sinh(z/2)$. In fact $P_n$ turns out to be a formal power series in all its variables, invariant with respect to the simultaneous action of the symmetric group $\mathfrak{S}_n$ on $(a_1,\dots,a_n)$ and $(x_1,\dots,x_n)$, see~\cite[Remarks 1.5 and 1.6]{BSSZ}.
Buryak finds that the $n$-point functions $F_n$ have the following Gaussian-integral representation:
\begin{align*}
F_n(\textbf{x}) =
\frac{e^{ \frac{p_3(\textbf{x})}{24} }}{e_1(\textbf{x}) \prod_{j=1}^{n} \sqrt{2 \pi x_j} } \int_{\mathbb{R}^n} \left[\prod_{j=1}^{n} e^{ - \frac{a_j^2}{2 x_j} } da_j \right] P_n(i\textbf{a}; \textbf{x}),
\end{align*}
where $i \textbf{a} = (\sqrt{-1}a_1, \dots, \sqrt{-1}a_n).$ and $p_i$ are power sums. Let us recall that the usual convention for the unstable cases $(g,n) =  (0,1)$ and $(0,2)$ is the following:
$$
\II 01 \frac{1}{1-x\psi_1} = \frac{1}{x^2}, \qquad \qquad \II 02 \frac{1}{(1 - x\psi_1)(1 - y\psi_2)} = \frac{1}{(x+y)}.
$$

\subsection{Okounkov formula} Define the function $\mathcal{E}(x_1,\dots,x_n)$ as
\begin{align*}
 \mathcal{E}(\textbf{x}) \coloneqq \frac{e^{ \frac{p_3(\textbf{x})}{12}}}{\prod_{j=1}^{n} \sqrt{4\pi x_j} } \int_{\mathbb{R}_{\geq 0}^n} \left[\prod_{j=1}^{n} ds_j \right]
 \exp \left({-\sum_{j=1}^n \left(\frac{(s_j-s_{j+1})^2}{4x_j} + \frac{(s_j+s_{j+1})x_j}{2} \right)}\right),
\end{align*}
where $s_{n+1}$ denotes $s_1$. Define the $\mathfrak{S}_n$-invariant function
\begin{align*}
\mathcal{E}^{\circlearrowleft}(x_1,\dots,x_n) \coloneqq \sum_{\sigma\in \mathfrak{S}_n/\mathbb{Z}_n} 
\mathcal{E}(x_{\sigma(1)},\dots,x_{\sigma(n)}).
\end{align*}
\noindent
Let $\Pi_n$ be the set of all partitions of $\{1,\dots,n\}$ into a disjoint union of unordered subsets $\sqcup_{j=1}^\ell I_j$, for all $\ell=1,2\dots,n$. Let $x_I\coloneqq \sum_{j\in I} x_j$, $I\subset \{1,\dots,n\}$. Finally, define
\begin{align*}
\mathcal{G}(\textbf{x})\coloneqq \sum_{\sqcup_{j=1}^\ell I_j \in \Pi_n} (-1)^{\ell+1} \mathcal{E}^{\circlearrowleft}(x_{I_1},\dots,x_{I_\ell}), \qquad 
F_n(\textbf{x}) \coloneqq  \frac{(2\pi)^{n/2}}{\prod_{j=1}^n\sqrt{x_j}}
\mathcal{G}\left(\frac{\textbf{x}}{2^{1/3}}\right).
\end{align*}

\subsection{Relation between Buryak and Okounkov formulae}
There is no obvious equality between the two formulae for $F_n$ obtained by Buryak and by Okounkov. A technical combinatorial argument proving their equality directly was achieved in \cite{AIS_OkBur}.

\subsection{Brezin-Hikami formula}
It is worth mentioning that Brezin and Hikami \cite{BH} derive formulae for the $n$-point functions $F_n^{(p)}$ for the intersection numbers of the moduli space of curves with a $(p-1)$-spin structure via Gaussian random matrix theory in the presence of an external matrix source, which restricts to $F_n$ for $p=2$, although their work focuses on the explicit investigation of higher $p$ for $n=1,2$. 

\subsection{Liu-Xu formula}
Dijkgraaf-Verlinde-Verlinde \cite{DVV} have recast \cref{thm:WK} in terms of Virasoro constraints applied to the partition function and shown that the two statements are in fact equivalent (which is in general not the case).
Liu and Xu \cite[Theorem 1.2, Corollary 2.2]{LiuXu_npoint}, exploiting the structure of the Virasoro constraints, have constructed recursively a solution of the $n$-point function (and its normalized version), which we now recall.

\begin{theorem}\label{thm:LiuXuFG}
For $n\geq2$, the $n$-point function takes the form
\begin{equation*}
F_n(x_1,\dots,x_n)=\sum_{g=0}^{\infty} \sum_{r = 0}^g \frac{(2r+n-3)!!}{12^{g-r}(2g+n-1)!!}S_r(x_1,\dots,x_n)e_1^{3g - 3r},
\end{equation*}
where $S_r$ is the homogeneous symmetric polynomial of degree $3r - 3 + n$ defined recursively by
\begin{align*}
S_r(x_1,\dots,x_n)
=
\frac{1}{2e_1}\sum_{\underline{n}=I\coprod
J}e_1(I)^2e_1(J)^2\sum_{r'=0}^r
A_{r', |I|}(x_I)A_{r-r', |J|}(x_J),
\end{align*}
where $I,J\ne\emptyset$ and $e_1(I)$ is $e_1$ evaluated in the variables $x_i$ for $i \in I$.
Equivalently, the normalized $n$-point function takes the form
\begin{align*}
G_n(x_1,\dots,x_n)&=
\sum_{g=0} \sum_{r = 0}^g \frac{(2r+n-3)!!}{12^{g - r}(2g+n-1)!!}P_r(x_1,\dots,x_n) (e_1^3 - p_3)^{g - r}
\\
&=\sum_{g=0} \sum_{r = 0}^g \frac{(2r+n-3)!!}{4^{g - r}(2g+n-1)!!}P_r(x_1,\dots,x_n) (e_1e_2 - e_3)^{g - r}
\end{align*}
where $P_r$ is the homogeneous symmetric polynomials of degree $3r - 3 + n$ defined recursively by
\begin{align*}
P_r(x_1,\dots,x_n)
=
\frac{1}{2e_1}\sum_{\underline{n}=I\coprod
J}e_1(I)^2e_1(J)^2\sum_{r'=0}^r
B_{r', |I|}(x_I)B_{r-r', |J|}(x_J),
\end{align*}
where $I,J\ne\emptyset$, $\underline{n}=\{1,2,\ldots,n\}$ and $B_g(x_I)$ denotes the degree $3g+|I|-3$ homogeneous
component of the normalized $|I|$-point function
$G_n(x_{k_1},\dots,x_{k_{|I|}})$, where $k_j\in I$.
\end{theorem}


\section{Structure imposed by Virasoro constraints}\label{sec:virasoro}
In this section we discuss the implications of the first two 
Virasoro constraints,  the so called string and dilaton equations,
on the coefficients \( C_{g,n}(\lambda) \) 
appearing in the elementary symmetric 
polynomial decomposition of \( A_{g,n} \)
\eqref{eq:AgnInCgn}.
We show how the string equation allows one to
drop the index \( n \) and define uniquely coefficients \( C_{g}(\lambda) \)
that appear in the formula of \( A_{g,n} \) for any \( n \). 

A careful analysis also puts a constraint on the size of 
the partition,  after the first row of the partition is excluded.
Then we use the dilaton equation to show the polynomiality behaviour of
\( C_{g}(\lambda) \) in terms of the length of the first row \( \lambda_{1} \).
This is enough to compute all coefficients $C_g$ for fixed \( g \) by computing the case in which $n$ is sufficiently large.  In particular, this allows us to check the expected vanishing up to genus $7$,  and therefore prove the conjecture in those cases.  Also, the recursion of the $C_g$ obtained by the dilaton equation provides new formulae for $A_{g,n}$ in genus $2,3,4$.

For higher Virasoro, we analyse the formula of Liu and Xu in light of our conjecture, providing a restatement of Conjecture \ref{conj:main} used in the section to prove small $n$ cases.

\subsection{Elementary symmetric functions}

For a set of variables $\textbf{x} = \{x_1, \dots, x_n\}$, let $e_i$ indicate
the elementary symmetric polynomials. 
Observe that:
\begin{align}
	\label{eq:elementary1}
	e_{k}(x_{1} ,\ldots,x_{n+1})\big|_{x_{n+1} = 0}
	&=
	\begin{cases}
		0 & \text{if} \quad k = n+1\\
		e_{k}(x_{1} ,\ldots, x_{n}) &  \text{otherwise.}
	\end{cases}
\end{align}
Moreover
\begin{equation}
\frac{\partial}{\partial x_{n+1}}e_{k}(x_{1} ,\ldots,x_{n+1})
=
\frac{\partial}{\partial x_{n+1}}e_{k}(x_{1} ,\ldots,x_{n+1})\big|_{x_{n+1} = 0}
=
e_{k-1}(x_{1} ,\ldots, x_{n}).
\label{eq:elementary2}
\end{equation}
\subsection{The Virasoro contraints} 
A careful analysis of the KdV hierarchy allows to restate \cref{thm:WK} in terms of the Virasoro algebra, as it was shown by Dijkgraaf, Verlinde and Verlinde \cite{DVV}. More precisely, \cref{thm:WK} is equivalent to the data of certain particular infinite sequence of operators $\mathcal{L}_m$ in the $\textbf{t}$, at most quadratic, such that
\begin{equation}
\mathcal{Z}^{\textup{WK}} := e^{F^{\textup{WK}}}, \qquad \qquad \mathcal{L}_{m}.{\mathcal{Z}}^{\textup{WK}} = 0, \;\; m \geq -1, \qquad \qquad [\mathcal{L}_m, \mathcal{L}_n] = (m-n)\mathcal{L}_{m+n}.
\end{equation}
\noindent
The first and the second Virasoro constraints, respectively, translate into \textit{string} and \textit{dilaton} equations:
\begin{align}
	 \mathcal{L}_{-1}.\mathcal{Z}^{\textup{WK}} = 0 & \quad \iff \quad \cor{ \tau_{d_{1}} \ldots \tau_{d_{n}} \tau_{0} }_{g} =
	\sum_{i=1}^{n} 
	\cor{ \tau_{d_{1}}\ldots \tau_{d_{i}-1}\ldots \tau_{d_{n} }}_{g},
	\label{eq:stringtau}
	\\
	 \mathcal{L}_0.\mathcal{Z}^{\textup{WK}} = 0 & \quad \iff \quad \cor{ \tau_{d_{1}} \ldots \tau_{d_{n}} \tau_{1} }_{g} =
	(2g-2+n)
	\cor{ \tau_{d_{1}} \ldots \tau_{d_{n}}}_{g}.
	\label{eq:dilatontau}
\end{align}
For $m > 0$, the Virasoro constraint $\mathcal{L}_{m}.\mathcal{Z}^{\textup{WK}} = 0$ takes the form
\begin{align*}
	(2m+3)!! \cor{\tau_{d_{1}} \ldots \tau_{d_{n}} \tau_{m+1}}_{g} 
	&= \sum_{i} \frac{(2d_{i} + 2 m + 1 ) !!}{(2 d_{i} -1 )!!} 
	\cor{\tau_{d_{1}} \ldots \tau_{d_{i} + m} \ldots \tau_{d_{n}}}  
	\\
	&+ \frac{1}{2} \sum_{a+b = m-1} 
	(2a+1)!!(2b+1)!! 
	\cor{\tau_{d_{1}} \ldots \tau_{d_{n}} \tau_{a} \tau_{b}}_{g-1} 
	\\
	&+ \frac{1}{2} \sum_{\substack{a + b = m-1 \\ I \sqcup J = \{1 ,\ldots, n\} \\ g_{1} + g_{2} = g}} 
	(2a+1)!!(2b+1)!! 
	\cor{\tau_{d_{I}} \tau_{a}}_{g_{1}} \cor{\tau_{d_{J}} \tau_{b}}_{g_{2}}  
\end{align*}

\subsection{String equation}

In the following we study the information string equation provides on the
coefficients $C_{g,n}(\lambda)$. It is straightforward to rewrite the string
equation \eqref{eq:stringtau} in terms of the amplitudes, which reads:
\begin{align}
	A_{g,n+1}(x_{1} ,\ldots,x_{n +1}) \big|_{x_{n+1}=0}
	&=
e_1 A_{g,n}(x_{1} ,\ldots,x_{n}). \label{eq:stringA}
\end{align} 
\noindent
Let us substitute our decomposition for \( A_{g,n} \) \eqref{eq:AgnInCgn} into the string
equation \eqref{eq:stringA}, and apply 
\eqref{eq:elementary1}, obtaining
\begin{equation}
	\sum_{\substack{|\lambda| \leq 3g-2+n \\ 2 \leq \lambda_{i} \leq n}} 
	C_{g,n+1}(\lambda)
	e_{\lambda}
	e_{1}^{3g-2+n-|\lambda|} 
	=
	\sum_{\substack{|\lambda| \leq 3g-3+n \\ 2 \leq  \lambda_{i} \leq n}} 
	C_{g,n}(\lambda)
	e_{\lambda}
	e_{1}^{3g-3+n-|\lambda|+1} .
\end{equation}
We then split the LHS in partitions with 
\( | \lambda | = 3g-2+n \) and partitions with \( | \lambda | \leq 3g-3+n \), we obtain
\begin{multline}
	\sum_{\substack{|\lambda| = 3g-2+n \\ 2 \leq \lambda_{i} \leq n }} 
	C_{g,n+1}(\lambda)
	e_{\lambda}
	+
	\sum_{\substack{|\lambda| \leq  3g-3+n \\ 2 \leq \lambda_{i} \leq n }} 
	C_{g,n+1}(\lambda)
	e_{\lambda}
	e_{1}^{3g-2+n-|\lambda|} 
	\\
	=
	\sum_{\substack{|\lambda| \leq 3g-3+n \\ 2 \leq \lambda_{i} \leq n }} 
	C_{g,n}(\lambda) 
	e_{\lambda} 
	e_{1}^{3g-2+n-|\lambda|}.
\end{multline}
We can match 
coefficients term by term:
\begin{align}
	\intertext{$\forall n \geq 1 $}
	C_{g,n+1}(\lambda) &= C_{g,n}(\lambda)	
	&&\text{ if }
	|\lambda| \leq 3g-3+n
			   &&\text{ and } 2 \leq \lambda_{i} \leq n 
			   \quad i = 1 ,\ldots, \ell(\lambda).
	\label{eq:string1}
	\\
\intertext{$\forall n \geq 2$}
	C_{g,n}(\lambda) &= 0 
	&&\text{ if }
	|\lambda| = 3g-3+n
			 &&\text{ and } 2 \leq \lambda_{i} < n 
			 \quad i = 1 ,\ldots, \ell(\lambda).
	\label{eq:string2}
\end{align}
As a consequence, we have the following lemma.
\begin{lemma}
	\label{lem:1}
	For all partitions \( \lambda \neq \emptyset \) 
	with \( |\lambda| \leq 3g-3+n \) 
	and \( 2 \leq \lambda_{i} \leq n \), \( i = 1 ,\ldots, \ell(\lambda) \),
	\begin{equation}
		C_{g,n}(\lambda) = 0  \qquad \text{ if } \qquad  |\lambda|- \lambda_{1} > 3g-3
		\label{eq:lem1}
	\end{equation}
\end{lemma}
\begin{proof}
	Let us fix \( g \) and \( n \) and 
	pick any nonempty partition 
	\( \lambda \) with \( \lambda_{i} \geq 2 \)
	and in which \( \lambda_{1} \leq n \).
	We want to apply several times equation \eqref{eq:string1},
	each time decreasing the value of \( n \) down to some \( n' \leq n \),
	either 
	reaching down to \( |\lambda| = 3g-3 + n' \)
	or 
	 \( \lambda_{1} = n' \).
	The first case requires \( \Delta_{1} = 3g-3+n - |\lambda|  \) steps
	and the second case requires \( \Delta_{2} = 
	 n - \lambda_{1}\) steps.
	If \( |\lambda| - n' = 3g-3 \) is reached first,
	but \( \lambda_{1} < n' \), then 
	by equation \eqref{eq:string2} we have
	\begin{equation}
		C_{g,n}(\lambda) = C_{g,n'}(\lambda) = 0.
	\end{equation}
	This will be the case whenever \( \Delta_{1} < \Delta_{2}  \),
	which is true whenever \( |\lambda| - \lambda_{1} > 3g-3 \).
\end{proof}
\begin{corollary}
	The string equation implies that Conjecture \ref{conj:main} holds true for \( g \leq 2 \).
\end{corollary}
\begin{proof}
	The genera \( g = 0,1 \) were already discussed in the introduction,
	and are again verified here. For \( g = 0 \)
	there are no nonempty partitions 
	and for \( g = 1 \) there is just \( \lambda = (\lambda_{1}) \).
	For \( g = 2 \) we require \( |\lambda| - \lambda_{1} \leq 3 \)
	and the only nonempty partitions are 
	 \( (\lambda_{1}), (\lambda_{1} ,2) \),
	and \( (\lambda_{1} , 3) \).
\end{proof}

\begin{remark} \label{rmk:bound}
Equation \eqref{eq:string1} suggests that we could
define coefficients \( C_{g}(\lambda) \) corresponding
to \( C_{g,n}(\lambda) \) for any \( n \) large enough.  If \( n \) is too small,
the corresponding monomial 
\(e_{\lambda} e_1^{3g-3+n - |\lambda|}\)
does not appear in \( A_{g,n} \), and \( C_{g,n}(\lambda) \)
is formally zero.
Therefore, we want to define \( C_{g}(\lambda) \) as the 
common value for all 
non trivial \( C_{g,n}(\lambda) \)'s.
Those can be computed whenever \( n \geq \lambda_{1} \) 
and \( |\lambda| \leq 3g-3+n \).
In particular, for non empty partitions we can always set 
\begin{equation}
	C_{g}(\lambda)
	=
	C_{g,n_0}(\lambda)
\end{equation}
where \( n_{0} = \max \{ \lambda_{1} , |\lambda| - (3g-3)\}  \).
Remark however that 
partitions with \( |\lambda| - (3g-3) > \lambda_{1} \) 
have vanishing coefficient according to lemma \ref{lem:1},
so the non-trivial coefficients are in general
\( C_{g}(\lambda) = C_{g, \lambda_{1}}(\lambda) \).
When \( \lambda = \emptyset \), the
coefficient \( C_{g,n}(\emptyset) \) will appear
in \( A_{g,n}(\textbf{x}) \) as long as 
\( 3g-3+n \geq 0 \), so for \( g = 0 \) we
can pick \( C_{0}(\emptyset) = C_{0,3}(\emptyset) \),
and for \( g \geq 1 \) we set \( C_{g}(\emptyset) = C_{g,1}(\emptyset ) \).
\end{remark}

\begin{corollary}\label{cor:33}
\begin{equation}
	A_{g,n}(\textbf{x}) 
	=
	\sum_{\substack{|\lambda| \leq 3g-3+n \\ 
			|\lambda| - \lambda_{1} \leq 3g-3\\
	\lambda_{i} \geq 2 }}
	C_{g}(\lambda) \; e_{\lambda}  \; e_{1}^{3g-3+n-|\lambda|}.   
	\label{eq:AgnInCG}
\end{equation}
\end{corollary}
In appendix
\ref{sec:Dgnappendix} we give a table of coefficients
\( D_{g,n}(\lambda \sqcup (1)^{3g-3+n})  = C_{g,n}(\lambda) \) 
for \( g = 3 \) and \( n \leq 5 \) to illustrate explicit checks
of relations \eqref{eq:string1}, \eqref{eq:string2} and \eqref{eq:lem1}.
\subsection{Comparison between the conjecture and the string equation constraint}

Let us comment on how much conjecture \ref{conj:main}
actually constraints the number of terms in the elementary symmetric polynomial
decomposition of \( A_{g,n} \),  compared to the constraint imposed by the string equation.

From the string equation and \eqref{eq:AgnInCG}, we see
that partitions \( \lambda = (\lambda_{1} )\sqcup\mu  \) can have at 
most  \( \ell(\mu) =  \floor{\frac{3g-3}{2}}\).
In conjecture \( \ref{conj:main} \) we claim that the 
maximum length is actually \( \ell(\mu) = g-1 \).
\begin{itemize}
\item Let us denote by \( \mathcal{Q}_{g} \) the
set of partitions
\( \mu \) with \( |\mu| \leq 3g-3 \) 
and \( \mu_{i} \geq 2 \).
\item Let us denote by \( \mathcal{Q}_{g}^* \) 
the set \( \mathcal{Q}_{g} \) under the further constraint \( \ell(\mu) \leq g-1 \).
\item Let us denote by \( \mathcal{Q}_{g}^{**} \) 
the set \( \mathcal{Q}_{g} \) under the further constraint \( \ell(\mu) \geq g-1 \).
\end{itemize}

One may ask how much bigger is 
\( \mathcal{Q}_{g} \) compared to \( \mathcal{Q}_{g}^* \).  The first observations is that each partition in $\mathcal{Q}_g \setminus \mathcal{Q}_g^{*}$ contains a block of the form $(2,2,2)$.  By considering the bijection that removes such block we get a bijection between
$
\mathcal{Q}_g \setminus \mathcal{Q}_g^{*} \leftrightarrow \mathcal{Q}_{g-2}^{**},
$
and therefore an equality of their cardinalities.  This way we have shown that the conjecture \ref{conj:main} discards from $\mathcal{Q}_g$ a subset of $\mathcal{Q}_{g-2}$.  The table below shows the cardinalities of these sets for small values of $g$.

\vspace{20pt}
\begin{center}
\begin{tabular}{ |c|c|c|c|c|c|c|c|}
 \hline
 $g$ & 6 & 9 & 12 & 15 & 18 & 21  \\
 \hline
 $|\mathcal{Q}_g|$  	& 176 & 1575 & 10143 & 53174 & 239943 & 966467 \\
 $|\mathcal{Q}_g^{*}|$ & 167 & 1528 & 9973 & 52649 & 238521 & 962922  \\
 $|\mathcal{Q}_{g-2}^{**}|$ & 9 & 47 & 170 & 525 & 1422 & 3545\\
 \hline
\end{tabular}
\end{center}
\vspace{20pt}

\noindent
We see that our conjecture only discards a fraction of the total amount of partitions considered.  We make this point to emphasize that
the strength of the conjecture does not rely in the vanishing of
a large amount of terms out of the elementary symmetric decomposition of
\( A_{g,n} \),  but rather in unveiling a peculiar and unexpected constraint these partitions obey. It would be ideal to have a geometric reason behind this vanishing.
We will see an application of this constraint in section \ref{sec:gdimintegral}.

\subsection{Dilaton equation}
It it straightforward to rewrite the dilaton equation \cref{eq:dilatontau} in terms of the amplitudes $A_{g,n}$, obtaining
\begin{equation}
	\frac{\partial}{\partial x_{n + 1}}
	A_{g,n+1}(x_{1},\ldots,x_{n+1})\big|_{x_{n+1} = 0}
	= (2g-2+n) A_{g,n}(x_{1},\ldots, x_{n}). \label{eq:dilat}
\end{equation}
Let us now employ the expansion of \( A_{g,n} \) 
which we have refined in the previous section \eqref{eq:AgnInCG},
and compute the left hand side with \eqref{eq:elementary2}:
\begin{multline}
	\sum_{|\lambda| \leq 3g-2+n} 
	C_{g}(\lambda) \sum_{j = 1}^{l(\lambda)} 
	e_{\lambda_{j} -1} \prod_{i \neq j} 
	e_{\lambda_{i}} 
	e_{1}^{3g-2+n- |\lambda|} \\
	+
	\sum_{ |\lambda| \leq 3g-2+n} 
	C_{g}(\lambda) e_{\lambda} (3g-2+n- |\lambda| ) e_{1}^{3g-3+n - |\lambda|} 
	\\
	\overset{}{=} 
	(2g-2+n)
	\sum_{|\lambda| \leq 3g-3+n} 
	C_{g}(\lambda) e_{\lambda} e_{1}^{3g-3+n- |\lambda|}
\end{multline}
To extract an identity on \( C_{g}(\lambda) \) we
must first reconstruct the elementary symmetric polynomial basis
in the first term on the left hand side. We have
two cases to consider: 
\begin{enumerate}
\item when \( \lambda_{j} = 2  \), 
then \( e_{\lambda_{j} -1} = e_{1} \) 
comes out of \( e_{\lambda} \) 
and we can relabel what is left
as a new partition \(\lambda'\) , 
which carries
the coefficient associated to
\( \lambda ' \sqcup (2) \). 
There
are \( m_{2}(\lambda) \) 
ways to
remove a row of length 2 from the partition \(\lambda\) 
to get \( \lambda' \),
so it comes
with multiplicity
\( m_{2}(\lambda') +1 \).

\item when \( \lambda_{j} > 2 \) and thus
\( \lambda_{j} -1 > 1 \),  we relabel it into a new partition \(\lambda'\).  This term
carries the coefficients of
all the partitions from which we
can obtain \( \lambda' \) by removing one block
and then reordering the rows (those partitions can obviously be obtained starting
from \(\lambda'\) by adding one block, but they come
with a different multiplicity than if we just counted the
ways to construct
them from \( \lambda' \) in all possible ways of adding one block).
 
We can describe the multiplicity in the following way. Let \( \lambda + (1)_{i} \) be the partition obtained
by adding to \(\lambda\) a block at the $i$-th row (and then reordering).
It is enough to consider \( \lambda + (1)_{i} \) 
for \( i = 1 \) and whenever \( \lambda_{i} < \lambda_{i-1} \)
(those cases that do not require reordering the rows).
Let \( \text{step}(\lambda) \subset \{1 ,\ldots, \ell(\lambda)\} \) 
represent this set of steps. Let us count in how many ways can we recover \( \lambda \) from
\( \lambda + (1)_{i} \) by removing a single block.
This is the number of rows of length 
\( \lambda_{i} + 1 \) in \( \lambda + (1)_{i} \),
which is the same at the number of rows
of length \( \lambda_{i} + 1 \) in \( \lambda \) plus one. Therefore we have:
\end{enumerate}
\begin{multline}
	\sum_{|\lambda| \leq 3g-3+n} 
	\sum_{j \in \text{step}(\lambda)}
	(m_{\lambda_{j} +1}(\lambda) + 1)
	C_{g}(\lambda+(1)_{j}) 
	e_{\lambda}
	e_{1}^{3g-3+n- |\lambda|} \\
	+
	\sum_{|\lambda| \leq 3g-4+n} 
	(m_{2}(\lambda) + 1) 
	C_{g}(\lambda \sqcup (2)) e_{\lambda} 
	e_{1}^{3g-3+n-|\lambda|} \\
	+
	\sum_{|\lambda| \leq 3g-3+n} 
	(g - |\lambda| ) 
	C_{g}(\lambda ) e_{\lambda} e_{1}^{3g-3+n - |\lambda|} 
	=
	0.
\end{multline}
Matching coefficients, we obtain
\begin{equation}
	\sum_{j \in \text{step}(\lambda)}
	(m_{\lambda_{j} +1}(\lambda) + 1)
	\,
	C_{g}(\lambda + (1)_{j})
	+
	(m_{2}(\lambda) + 1) 
	\,
	C_{g}(\lambda \sqcup (2))
	+
	(g - |\lambda| ) 
	\,
	C_{g}(\lambda ) 
	=
	0.
	\label{eq:Cgdilaton}
\end{equation}

\subsection{Solving the recursion}
Equation \eqref{eq:Cgdilaton} is a recursion on \( \lambda \) for fixed \( g \):
the \( C_{g}(\lambda) \) for smaller partitions are related to the
\( C_{g}(\lambda) \) for bigger partitions.  However,  whenever the 
partition minus its first row has size exceeding \( 3g-3 \),
the coefficients vanish.  To make use of this, let us split partitions as
\( \lambda = (\lambda_{1} ) \sqcup ( \lambda_{2} ,\ldots, \lambda_{\ell(\lambda) })
\coloneqq (k) \sqcup \mu\) and let us denote
\begin{equation}
	C_{g}(\lambda) = C_{g}(k, \mu).
\end{equation}
\begin{figure}
\begin{center}\includegraphics[width = 6cm]{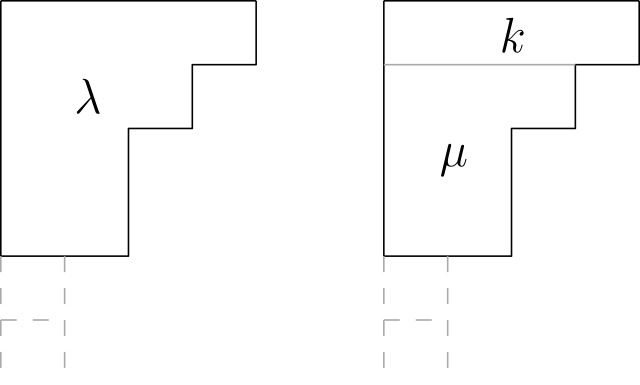}\end{center}
\caption{From \( \lambda \) to (\( k \),\( \mu \)) }
\end{figure}

\noindent
This notation only makes sense for \( k \geq \mu_{1} \),
and from lemma \ref{lem:1} we must have \( |\mu| \leq 3g-3 \).
For the empty partition \( \lambda = \emptyset \)
we keep the notation \( C_{g}(\emptyset) \),
whereas \( C_{g}(k, \varnothing ) \) indicates a partition
of a single row of length \( k \), \( k \geq 2 \).
The coefficient \( C_{g}(\emptyset) \)
can be computed for all \( g \) from the known formula for the
1-point function, and it is given in section \ref{sec:n1}.
Let us therefore focus on the remaining partitions
that can we decomposed at \( (k, \mu ) \) for \( k \geq 2 \).
Let us rewrite equation \eqref{eq:Cgdilaton} in terms of pairs \( (k, \mu ) \).
For \( \lambda = \emptyset \),
\begin{equation}\label{eq:Cg2empty}
	C_{g}(2,\varnothing)
	=
	- g \, C_{g}(\emptyset ) .
\end{equation}
And for all other partitions,
\begin{multline}
	C_{g}(k+1, \mu) 
	+\sum_{j \in \text{step}(\mu)}
	(m_{\mu_{j} +1}(\mu) + 1)
	C_{g}(k, \mu+(1)_{j}) 
	\\
	+ \delta_{k, \mu_{1}+1} C_{g}(k, \mu + (1)_{1} )
	+ \delta_{k, \mu_{1}} C_{g}(k, \mu + (1)_{\iota}) 
	\\[10pt]
	+ ( m_{2}(\mu) + \delta_{k,2}+ 1) C_{g}(k, \mu \sqcup (2)) 
	+ (g- k -|\mu|)C_{g}(k, \mu)
	=0.
\end{multline}

\noindent
Here, \( \iota \) is the first index such that \( \mu_{1} = \mu_{\iota} + 1 \) 
(which could possibly not
exist, and in that case this extra term can be ignored). 
If terms of the form \( C_{g}(k, \mu') \) with \( \mu'_{1}>k  \) appear in the sum, we discard them
by convention.  This is equivalent to say that 
if \( \mu_{1} = k \) then \( 1 \) should
not be included in \( \text{step}(\mu) \).

Let us impose \( k \geq g - |\mu| +1 \) and define
\begin{equation}
	C_{g}(k, \mu) = (-1)^{\mu + l(\mu)} ( k + |\mu|-g-1)! \, \mathcal{C}_{g}(k, \mu).
	\label{def:mathcalC}
\end{equation}
The relation now reads
\begin{multline}
	\label{eq:mathcalC}
	\mathcal{C}_{g}(k+1, \mu) 
	-
	\mathcal{C}_{g}(k, \mu) 
	=\\
	\sum_{j \in \text{step}(\mu)} 
	(m_{\mu_{j} +1}(\mu) + 1)
	\mathcal{C}_{g}(k, \mu + (1)_{j}) 
	+ (k+|\mu| - g + 1 )( m_{2}(\mu) + 1) \mathcal{C}_{g}(k, \mu \sqcup (2)) 
	\\
	+ (|\mu| - g + 3 ) \delta_{k,2} C_{g}(k, \mu \sqcup (2))
	+ \delta_{k, \mu_{1}+1} \mathcal{C}_{g}(k, \mu + (1)_{1} )
	+ \delta_{k, \mu_{1}} \mathcal{C}_{g}(k, \mu + (1)_{\iota}).
\end{multline}
Let us adopt the convention that \( \mu_{1} = 2 \) for \(\mu = \varnothing\).
We have the following lemma.
\begin{lemma}
	\label{lem:2}
	Let \( \mu \in \mathcal{Q}_{g} \).
	For \( | \mu| = 3g-3 \) ,  the
	coefficients \( \mathcal{C}_{g}(k, {\mu}) \)  are
	constant functions of \( k \) whenever \( k \geq {\mu_{1}} \). 
	For \( |\mu| < 3g-3 \),  the
	coefficients \( \mathcal{C}_{g}(k, \mu) \) for \( k > 3g-3 - |\mu| 
	+ {\mu_1} \)
	are polynomials in \( k \) of degree \( 3g-3- |\mu| \).
\end{lemma}
\begin{proof}
	Consider \eqref{eq:mathcalC} for a partition \( \mu \) with \( \mu = 3g-3 \):
	the RHS vanishes,  which implies that \( C_{g}(k, \mu) \) is a constant function of \( k \). This is true for all \( k \geq \mu_{1} \),
	and for \( \mu = \varnothing \) when \( g = 1 \) this is true for
	\( k \geq 2 \).

	When \(\mu < 3g-3\), we want to get rid of the \( \delta \) 
	terms, so we impose \( k > \mu_{1} + 1 \),
	(in the case \( \mu = \varnothing \) and \( g >1 \),
	we impose \( k \geq 3 \)).  Hence we have:
	\begin{multline}
	\mathcal{C}_{g}(k+1, \mu) 
	-
	\mathcal{C}_{g}(k, \mu) 
	=\\
	\sum_{j \in \text{step}(\mu)} 
	(m_{\mu_{j} +1}(\mu) + 1)
	\mathcal{C}_{g}(k, \mu + (1)_{j}) 
	+ (k+|\mu| - g + 1 )( m_{2}(\mu) + 1) \mathcal{C}_{g}(k, \mu \sqcup (2)).
	\label{eq:induct}
	\end{multline}
	Let us analyse what is needed in order to write \eqref{eq:induct} for all 
	partitions \( \mu' \) that can be constructed by
	adding blocks to \( \mu \), still with \( \mu' \leq 3g-3 \).
	This requires that we also impose all conditions 
	\( k > {\mu'_{1}} + 1 \)
	(except for the max size partition that will have constant
	coefficients).
	The most constraining partition for \( k \) will
	be the partition of size \( 3g-4 \)
	where we placed all the remaining blocks on the first row,
	giving \( {\mu'_{1}} = {\mu_{1}} + 3g-4 - |\mu| \).
	Therefore if \( k > 3g-3-|\mu|+\mu_1 \) we
	can impose all relations \eqref{eq:induct} simultaneously
	for \( \mu \) and all its ascending partitions.
	Now proceed by induction:
	if for all partitions \(\mu'\) with size \( |\mu'| > |\mu| \),
	\( \mathcal{C}_{g}(k, \mu') \) 
	are polynomials of degree \( 3g-3-|\mu'| \) for 
	\( k > 3g-3-|\mu'|+ {\mu'_{1}} \), then
	by \eqref{eq:induct},
	the difference \( \mathcal{C}_{g}(k+1, \mu) - \mathcal{C}_{g}(k, \mu) \) is
	a polynomial of degree \( 3g-3- | \mu| - 1 \)
	for \( k > 3g-3-|\mu| + {\mu_{1}} \),
	so \( \mathcal{C}_{g}(k, \mu) \) is itself a polynomial of degree 
	\( 3g-3-|\mu| \).
\end{proof}
\begin{corollary}
	For \( g \geq 1 \), for \( k > 3g-3-|\mu|+ {\mu_1} - \delta_{|\mu|,3g-3} \),
	\begin{equation}
		C_{g}(k, \mu) = ( k + |\mu| -g-1)! \, Q^{\mu}_{g}(k)
	\end{equation}
	where \( Q_{g}^{\mu} \) is a polynomial of degree \( 3g-3- |\mu| \).
	\label{lem:poly}
\end{corollary}
\begin{proof}
	We have to compare condition 
	\( \mathbb{A}: k > 3g-3-|\mu|+ {\mu_1} - \delta_{|\mu|,3g-3} \) 
	from the lemma above and condition \( \mathbb{B}: k > g-|\mu| \) required by
	the decomposition \eqref{def:mathcalC}.
	The difference is
	\begin{equation}
		2g-3+ {\mu_1} - \delta_{|\mu|,3g-3} .
	\end{equation}
	For \( g \geq 2 \) or for \( \mu_{1} \geq 2  \) this quantity is positive,  and hence \( \mathbb{A} \implies \mathbb{B} \).
	When \( g = 1  \) and \( \mu = \varnothing \), 
	\( \mathbb{A}: k > -1 \) and \( \mathbb{B}: k > 1 \), 
	although both conditions are taken under the constraint \( k \geq 2 \).  Therefore, whenever lemma \ref{lem:2} applies, it is possible to apply the decomposition 
	in the form of \eqref{def:mathcalC}.
\end{proof}
\begin{remark}
	Corollary \ref{lem:poly} is only a fraction of the information
	that we can extract from the dilaton equation. Building
	on that result, one could then decompose 
	\( C_{g}(k, \mu) \) in
	a suitable basis of polynomials,
\begin{equation}
	C_{g}(k, \mu) = \sum_{j = 0}^{3g-3-|\mu|} 
	\alpha_{j}^{(g)}(\mu) P_{j}^{\mu}(k),
\end{equation}
	and obtain a recursion on the coefficients
	\( \alpha_{j}^{(g)}(\mu) \) with initial values
	given by \( \alpha_{0}^{(g)}(\mu) \).
	While this distillates the initial data in a yet smaller
	set of coefficients, it is not so useful at present because we
	can only compute \( \alpha_{0}^{(g)}(\mu) \) 
	from the entire polynomial \( C_{g}(k, \mu) \).
\end{remark}

\begin{lemma} For any $g\geq 0$, we have
\begin{equation}
C_g(\emptyset) = 1, \qquad \qquad C_g(2,\emptyset) = -g.
\end{equation}
\begin{proof}
The cases for $g=0$ and $g=1$ are checked directly,  the closed formulae are known in the literature and given in \eqref{eq:A0n} and \eqref{eq:A1n}.  It is well-known that for $n=1$ one has $\tilde{A}_{g,1} = e_1^{3g - 2}$,  and therefore $C_g(\emptyset) = 1$.  Use the string equation backwards to exploit the $n-$independence in this degenerate case and prove that indeed $C_{g}(\emptyset) = 1$ in each $\tilde{A}_{g,n}$.  The  coefficient $C_g(2,\emptyset) = -g$ is given by equation \eqref{eq:Cg2empty}.
\end{proof}
\end{lemma}

\subsection{Compute \( A_{g,n} \) for small fixed \( g \)}
We now use corollary \ref{lem:poly} to compute, for fixed small \( g \),
\( C_{g}(\lambda) \) for all \( \lambda \).
We again denote by \( \mu \) the partition \( \lambda \) minus
its first row.  Recall that we require 
\( |\mu| \leq 3g-3 \) in order for \( C_{g}(\lambda) \) to 
be non-zero.
We can, for fixed \( g \) and \( \mu \), 
compute \( C_{g}(k,\mu) \) for the first few values of \( k \),
and use the polynomial structure discussed in corollary \ref{lem:poly},
then repeat this procedure for all \( \mu \in \mathcal{Q}_{g}\).

\subsubsection{\textit{The minimal amount of coefficients to determine all $C_g(\lambda)$}}
We have seen that
\( C_{g}(k,\mu) \) has a polynomial behaviour in $k$
starting at \( k_{0} = 3g-2-|\mu| + {\mu_{1}} - \delta_{|\mu|,3g-3}  \),
and the polynomial is of degree \( 3g-3-|\mu| \). 
Thus,  it is enough to compute the \( C_{g}(k_0,\mu) \) 
up to \( C_{g}(k_0 + (3g-3) - |\mu|) \) in order to fit the polynomial part of 
\( C_{g}(k, \mu) \) for all \( k \).
Since \( C_{g}(k,\mu)  \) appears first in
\( A_{g,k} \), this means we need to compute \( A_{g,n} \)
for sufficiently high \( n \),  namely \( A_{g,n_{\text{min}}(g)}  \) with
\begin{equation}
	n_{\text{min}}(g) = \max_{\mu \in \text{part}_{g}}  
	\{6g-5-2|\mu|+ {\mu_{1}} - \delta_{|\mu|,3g-3}\} .
\end{equation}
This maximum is obviously always achieved by the empty partition
\(\mu = \varnothing\) ,
therefore,
\begin{equation}
	n_{\text{min}}(g)  = 6g-3-\delta_{g,1}.
\end{equation}
To summarize, our method allows one to compute closed expressions of 
\( A_{g,n} \) for all \( n \) and fixed \( g \),  provided one is able to compute
\( A_{g,6g-3} \). 

Let us also mention that our main conjecture cannot
really help us in computing \( A_{g,n} \): while the conjecture restricts the set \( \mathcal{Q}_{g} \) of partitions to \( \mathcal{Q}^*_{g} \),  this restriction does not bound the polynomial degree required by our method.  On the contrary, the most demanding partition in term of polynomial degree is the partition of minimal length \( \mu = \varnothing \) (whereas the conjecture concerns long partitions).

\subsection{Verification of the conjecture for small \( g \)}
Let us discuss some of the consequences of corollary \ref{lem:poly}
on conjecture \ref{conj:main}.
Corollary \ref{lem:poly} applies to all partitions in \( \mathcal{Q}_{g} \)
regardless of their length, and in particular we can
compute \( C_{g}(k, \mu) \) for \( \mu \) with \( \ell(\mu) > g-1 \)
and show by the same polynomial argument that they are indeed zero for all \( k \),
proving the conjecture for fixed \( g \).
This is what we verify explicitly
in the formulas (\ref{eq:A2n}) and (\ref{eq:A3n}) for the genus 2 and 3 amplitudes
, and for \( g = 4 \) in appendix \ref{sec:Cgappendix}.

To check the conjecture for fixed \( g \), we do not
need to compute the entire \( n-\)point function \( A_{g,n} \),
but it is enough to apply the polynomial argument to \( C_{g}(k, \mu) \) 
for partitions in \( \mathcal{Q}_{g} \setminus \mathcal{Q}^*_{g} \).
This requires to know the \( n \) point function only
up to \( A_{g,n_{conj}(g)}  \) with
\begin{equation}
	n_{\text{conj}}(g) = \max_{\mu \in \mathcal{Q}_{g} \setminus 
	\mathcal{Q}_{g}^*}  
	\{6g-5-2|\mu|+ \mu_{1} - \delta_{|\mu|,3g-3}\} 
\end{equation}
This maximum is achieved by the partition \( (2)^{g} = (\overbrace{2 ,\ldots, 2}^{ g \text{ times}}) \) (and \(\mathcal{Q}_{g} \setminus \mathcal{Q}_{g}^*  \)
is empty for \( g \leq 2 \), so we can assume that \( g \geq 3 \)
and \( (2)^{g} \) is indeed present) and therefore we have
\begin{equation}
	n_{\text{conj}}(g)  = 2g-1.
\end{equation}
We can summarize this discussion as follows.
\begin{lemma}
	\label{lem:bound}
	For all \( g \geq 3 \), conjecture \ref{conj:main} holds for 
	all \( n \geq 1 \) if and only if it holds for \( n = 2g-1 \).
\end{lemma}

Based on this result we were able to check the conjecture
for \( g \leq 7 \).

\subsection{Higher Virasoro constraints}
In the previous section we have seen how to transfer from string and from dilaton equations useful information on the coefficients $C_g(\lambda)$. It is a natural question to ask what value can higher Virasoro constraints provide. Let us first observe that the Virasoro algebra commutation relations imply that the algebra is generated by, for instance, the operators $\mathcal{L}_{-1}$ and $\mathcal{L}_{2}$, but not by $\mathcal{L}_{-1}$ and $\mathcal{L}_{0}$. There seem to be therefore something left to be captured. 

We have tried directly to derive for each given $m > 0$ a recursion for the $C_g(\lambda)$. However, the recursions arising this way do not seem to be very suitable for practical use. Instead, we are going to employ the result of Liu-Xu \cite{LiuXu_npoint} to restate our main conjecture in terms of the homogeneous polynomials $P_r$ and $S_r$.

\begin{corollary}
	\label{cor:mainconj}
Let $g \geq 0$ and $n \geq 1$ be integer numbers such that $2g - 2 + n > 0$. The following three statements are equivalent. 
\begin{enumerate}
\item Conjecture \ref{conj:main}:
\begin{equation*}
  A_{g,n}(\textbf{x}) 
	=
	\sum_{\substack{|\lambda| \leq 3g-3+n \\ \lambda_{i} \geq 2, \;\; l(\lambda) \leq g}}
	\!\!\!\! C_{g}(\lambda) \; e_{\lambda}  \; e_{1} ^{3g-3+n-|\lambda|}.   
\end{equation*}
\item
\begin{equation*}
  S_{g}(\textbf{x}) 
	=
	\sum_{\substack{|\lambda| \leq 3g-3+n \\ \lambda_{i} \geq 2, \;\; l(\lambda) \leq g}}
	\!\!\!\! C'_{g}(\lambda) \; e_{\lambda}  \; e_{1} ^{3g-3+n-|\lambda|}.   
\end{equation*}
\item
\begin{equation*}
  P_{g}(\textbf{x}) 
	=
	\sum_{\substack{|\lambda| \leq 3g-3+n \\ \lambda_{i} \geq 2, \;\; l(\lambda) \leq g}}
	\!\!\!\! C''_{g}(\lambda) \; e_{\lambda}  \; e_{1} ^{3g-3+n-|\lambda|}.   
\end{equation*}
\end{enumerate}
\begin{proof} It follows from the homogeneous degree of $P_r$ and $S_r$, together with the fact that in the formula for $F_n$ and $G_n$ only monomials in elementary symmetric polynomials of degree $3$ ($e_1e_2$ and $e_3$) appear.
\end{proof}
\end{corollary}


\section{Verifying the conjecture for small $n$}\label{sec:n123}
This section contains the proof of the conjecture for $n=1,2,3$.
\subsection{The case $n$ = 1} 
\label{sec:n1}
A closed formula for the $1$-point function $F_1(x)$ was computed by Witten 
\cite{W91} and reads
\begin{equation}
F_1(x) = \frac{e^{\frac{x^3}{24}} }{x^2}= \sum_{g=0}^{\infty} \frac{x^{3g - 2}}{24^g g!}.
\end{equation}
Since $x = e_1(x)$ and $e_i = 0$ for $i>1$, we trivially obtain:
\begin{lemma} Conjecture \ref{conj:main} holds for $n=1$ and any $g \geq 1$. More precisely we have:
	\begin{equation}
A_{g,1}(x) = C_{g}(\emptyset) e_1^{3g - 2}, \qquad \qquad {C}_{g}(\emptyset) = \frac{1}{24^g g!}.
\end{equation}
\end{lemma}
\noindent
Note that the case $(g,n) = (0,1)$ is discarded as it is unstable, and corresponds to the unstable integral $$
\int_{\MMMbar_{0,1}} \frac{1}{1 - x \psi_1} = \frac{1}{x^2}.$$
\subsection{The case $n$ = 2} 
A closed formula for the $2$-point function $F_2(x_1, x_2)$ was computed by Dijkgraaf and reads in terms of the amplitudes
\begin{equation}
	F_{2}(x_{1} , x_{2})
	=
	e^\frac{p_3}{24}
	\frac{1}{e_1}
	\sum_{k \geq 0}
	\frac{k!}{2^k (2k+1)!} e_2^k e_1^k 
\end{equation}
Let us verify the main conjecture \ref{conj:main} by expanding $F_2$ purely in elementary symmetric polynomials.
\begin{proposition}  Conjecture \ref{conj:main} holds for $n=2$ and any $g \geq 1$. More precisely we have:
	\begin{equation}
		A_{g,2}(\textbf{x})= \sum_{m = 0}^{g} C_{g}((2)^m)
		e_{2}^{m} e_{1}^{3g-1-2m}, \qquad \qquad C_{g}((2)^m) = \frac{1}{24^g g!} \binom{g}{m} \frac{(-3)^m}{(2m + 1)}.
 \end{equation}
\begin{proof}
Let us first express the power sum in terms of the $e_i$. One has \( p_3 = e_{1}^{3}
- 3 e_{1} e_{2}\), as $e_3$ vanishes in two variables. Substituting we obtain
\begin{align}
	F_{2}(p_{1} , p_{2})
	&=
	\sum_{l} \frac{(e_{1}^{3} - 3 e_{1} e_{2})^{l}}{l!}
	\frac{1}{24^{l}}
	e_{1}^{-1}
	\sum_{k} \frac{k!}{(2k+1)!} \frac{1}{2^{k}} e_{1}^{k} e_{2}^{k}
	\\
	&=
	\sum_{k} \sum_{l} \sum_{j = 0}^{l}
	\frac{1}{l! 24^{l}} (-3)^{j}
	\binom{l}{j}
	\frac{k!}{(2k+1)!}
	2^{-k}
	e_{1}^{3l-1-2j+k} e_{2}^{j+k}
\end{align}
Now let us fix the degree to be \( 3l+3k-1 = 3g-1  \). Hence
\( g = l+k \), \( l = g-k \), and we have
\begin{align}
	A_{g,2} 	&=
	\sum_{k} \sum_{j = 0}^{g-k}
	\frac{24^{k}}{(g-k)! 24^{g}} (-3)^{j}
	{\binom{g-k}{j}}
	\frac{k!}{(2k+1)!}
	2^{-k}
	e_{1}^{3g-1-2j-2k} e_{2}^{j+k}.
\end{align}
Setting \( j+k = m \), with \( m = 0,...,g \) and \( k = 0 ,...,m \) we get
\begin{align}
	A_{g,2} 	&=
	\sum_{m} \sum_{k = 0}^{m}
	\frac{24^{k}}{(g-k)!
	24^{g}} (-3)^{m-k}
	{\binom{g-k}{m-k}}
	\frac{k!}{(2k+1)!}
	2^{-k}
	e_{1}^{3g-1-2m} e_{2}^{m}
	\\
			&=
	\sum_{m} \sum_{k = 0}^{m}
	\frac{1}{24^{g} g!} (-3)^{m}  {\binom{g}{m}}
	(-4)^k \frac{m!}{(m-k)!}  \frac{k!}{(2k+1)!}
	e_{1}^{3g-1-2m} e_{2}^{m}
\end{align}
Therefore we have
\begin{equation}
	C_{g}((2)^{m}) =
	\sum_{k = 0}^{m}
	\frac{1}{24^{g} g!} (-3)^{m}  {\binom{g}{m}}
	(-4)^k \frac{m!}{(m-k)!}  \frac{k!}{(2k+1)!}.
\end{equation}
It simply remains to show the combinatorial identity
\begin{equation}
	\frac{1}{2m+1} =
	\sum_{k = 0}^{m}
	(-4)^k \frac{m!}{(m-k)!}  \frac{k!}{(2k+1)!}.
\end{equation}
In fact this identity is a particular case  for $2r+n = 2$ of the one proved in \cite[Lemma 2.6]{LiuXu_npoint}. This concludes the proof of the proposition.
\end{proof}
\end{proposition}
\noindent
Note that the case $(g,n) = (0,2)$ is discarded as it is unstable.
\begin{remark}
Let us see how one could have recovered Dijkgraaf formula directly from \cref{thm:LiuXuFG}: since $F_1(x) = e^{x^3/24}/x^2$, then $G_1(x) = x^{-2}$, hence $B_{r,1}(x) = \delta_{r,0}/x^2$. We compute
$$
P_{r}(x_1, x_2) = \frac{1}{2e_1} 2x_1^2x_2^2 \frac{\delta_{r,0}}{x_1^2x_2^2} = \frac{\delta_{r,0}}{e_1}
$$
It is then enough to substitute this into the formula for $G_2 = e^{-p_3/24}F_2$ and solving for $F_2$ converting the double factorials into factorials.
\end{remark}

\subsection{The case $n$ = 3} 
The following formula for $P_{r}(x_{1} , x_{2} , x_{3}) $ was initially given by Zagier \cite{Zagier3point}. A proof can be found in \cite{LiuXu_npoint}.
\begin{align*}
	&P_{r}(x_{1} , x_{2} , x_{3}) =
	\\
	&= \frac{1}{4^{g} (2g+1)!!} 
	\frac{1}{e_{1}} \left( (x_{1} + x_{2} )^{r+1} (x_{1}x_{2} )^{r} 
	+(x_{2} + x_{3} )^{r+1} (x_{2}x_{3} )^{r}
	+(x_{1} + x_{3} )^{r+1} (x_{1}x_{3} )^{r}\right).
	\label{eq:P3}
\end{align*}

\begin{proposition} Conjecture \ref{conj:main} holds for $n=3$ and any $g \geq 0$. More precisely we have:
\begin{multline}
	4^{-r}(2r + 1)!! \; P_{r}(x,y,z) 
	=
	\\
	=
	-(3r+2) (-1)^{r+1} e_{3}^{r}
	+\sum_{k = 1}^{r}  \underset{t = 0}{\operatorname{Res}} 
	\frac{dt}{t^{r+1}} \frac{e_{3}^{r-k}}{k} 
	\left( r + \frac{t}{e_{1}} \right) \left( e_{1} - t\right)^{r} 
	\left( te_{2} - t^2 e_{1} + t^3 \right)^{k}
\end{multline}
\end{proposition}
We are going to prove the residue formula. Notice that it implies the conjecture for $n=3$ through \cref{cor:mainconj}, as in each summand we have at most $r-k$ factors of $e_3$ and at most $k$ factors of $e_2$.

\begin{proof}
Let us consider the normalized version $T_r$ of $P_r$:
\begin{align}
	T_{r}(x,y,z) &= \frac{1}{e_{1}} 
	\left( (x+y)^{r+1} (xy)^{r} + (x+z)^{r+1} (xz)^{r} + (y+z)^{r+1} (yz)^{r}  \right)
	\\
		     &= 
		     \frac{1}{e_{1}} \left( (xy)^{r} \sum_{j = 0}^{r+1} e_{1}^{r+1-j} (-1)^{j} z^{j}
		     \binom{r+1}{j} + \text{sym} \right) 
		     \\
		     &= 
		     (-1)^{r+1} e_{3}^{r}
		     +\left( (xy)^{r} \sum_{j = 0}^{r} e_{1}^{r-j} (-1)^{j} z^{j}
		     \binom{r+1}{j} + \text{sym} \right) 
\end{align}
\begin{align}
	\iff T_{r}(x,y,z) + (-1)^{r} e_{3}^{r} &=
		   \sum_{j = 0}^{r} e_{1}^{r-j} (-1)^{j} z^{j}(xy)^{r} 
		     \binom{r+1}{j} + \text{sym}
		     \\
					       &=
					       e_{3}^{r} 
		   \sum_{j = 0}^{r} e_{1}^{r-j} (-1)^{j}
		   \binom{r+1}{j} \left( \frac{1}{x^{r-j}} + \frac{1}{y^{r-j}} + 
		   \frac{1}{z^{r-j}}\right) 
\end{align}
We want to describe this \( p_{r-j}(1/x,1/y,1/z) \) factor.
Let,
\begin{align}
	\pi(t) &= (x-t)(y-t)(z-t)
	\\
	       &= e_{3} - t e_{2} + t^2 e_{1} - t^{3}
\end{align}
Consider the development at small \( t \),
\begin{align}
\log(\frac{\pi(t)}{e_{3} }) &= 
	\log(1- \frac{t}{x}) + 
	\log(1- \frac{t}{y}) + 
	\log(1- \frac{t}{z})
	\\
				   &= -\sum_{k \geq 1} \frac{1}{k} t^{k} 
				   \left( \frac{1}{x^{k}} + \frac{1}{y^{k}} + \frac{1}{z^{k}} \right) 
				   \end{align}
				   \begin{align}
				   \iff -m\,\underset{t = 0}{\operatorname{Res}}  
				   \frac{dt}{t^{m+1}} 
			\log(\frac{\pi(t)}{e_{3}})
			&= p_{m}(1/x,1/y,1/z) 
\end{align}
On the other hand,
\begin{align}
\log(\frac{\pi(t)}{e_{3} }) &= 
\log(1 - \frac{1}{e_{3}} (t e_{2} - t^2 e_{1} + t^{3})) \\
			    &=
			    -\sum_{k \geq 1} \frac{1}{k} e_{3}^{-k} 
			    (t e_{2} - t^2 e_{1} + t^3 )^{k} 
\end{align}
Therefore,
\begin{align}
	p_{r-j}(1/x,1/y,1/z) 
	&= 
	\sum_{k \geq 1} \underset{t = 0}{\operatorname{Res}} 
	\frac{dt}{t^{r-j+1}} \frac{r-j}{k} e_{3}^{-k} \left( te_{2} - t^2 e_{1} + t^3 \right)^{k}
	\\
	&= 
	\sum_{k = 1}^{r}  \underset{t = 0}{\operatorname{Res}} 
	\frac{dt}{t^{r-j+1}} \frac{r-j}{k} e_{3}^{-k} \left( t e_{2} - t^2 e_{1} + t^3 \right)^{k}
\end{align}
Please note that this is only valid for \( r-j \geq 1 \). We easily
compute \( p_{0}(1/x,1/y,1/z) = 3 \).
We find
\begin{multline}
	T_{r}(x,y,z) + (-1)^{r} e_{3}^{r} \\
					       =
	\sum_{k = 1}^{r}  \underset{t = 0}{\operatorname{Res}} 
	\frac{dt}{t^{r+1}} \frac{e_{3}^{r-k}}{k} 
	\left( te_{2} - t^2 e_{1} + t^3 \right)^{k}
	\left(
		   \sum_{j = 0}^{r-1} 
		   e_{1}^{r-j} (-t)^{j}
		   \frac{(r+1)!}{(r+1-j)! j!}
		   (r-j)
	   \right)
					     \\\hspace{10pt}+ 3 e_{3}^{r} (-1)^{r} (r+1)
\end{multline}
Remains to repack the sum over \( j \),
\begin{align}
		   &\sum_{j = 0}^{r} 
		   e_{1}^{r-j} (-t)^{j}
		   \frac{(r+1)!}{(r+1-j)! j!}
		   (r-j)
		   \\
		   &=
		   \sum_{j = 0}^{r} 
		   e_{1}^{r-j} (-t)^{j}
		   \left(
			(r+1)\frac{r!}{(r-j)! j!}
			-\frac{(r+1)!}{(r+1-j)! j!}
		   \right)
		 \\&=(r+1)(e_{1} - t)^{r} - \frac{1}{e_{1}} \left((e_{1} -t)^{r+1} - (-1)^{r+1} t^{r+1} \right)
		 \\
		   &= \left( r + \frac{t}{e_{1}} \right) \left( e_{1} - t\right)^{r} 
		   + (-1)^{r} t^{r+1}
\end{align}
Notice that he second term does not contribute in the residue. This concludes the proof of the proposition.
\end{proof}


\section{An equivalent conjecture in terms of $\Omega$ classes}\label{sec:omega}

$\Omega$ classes can be thought of as a collection $\Omega$ of cohomology classes which play a central role
in the moduli spaces of curves: they are involved in several ELSV formulae arising in the context of
Hurwitz theory \cite{LPSZ, Lewa, DL}, in the enumeration of Masur-Veech volumes
\cite{MVchen}, in the double ramification cycle \cite{JPPZ}, in the context of topological
recursion for the $2$-KP and the $2$-BKP hierarchy \cite{GKL}, Witten's class, and more. They can be defined as certain cohomological classes $\Omega_{g,n}(r,s; a_1, \dots, a_n)
\in H^*(\MMMbar_{g,n})$, arising from $r$ dimensional cohomological field
theories, which generalize the Chern class of the Hodge bundle in the context
of $r$-spin curves with a corresponding generalized Mumford formula for the
Chern characters. Their polynomiality in the parameters $r$, $s$ and $a_i$
remains in general a mystery. In this section we aim to partially shed light in
this direction: we use a result of Goulden-Jackson-Vakil together with a ELSV
formula of \cite{DL} to transfer conjecture \ref{conj:main} to a statement
involving the polynomiality of the \text{intersection} of $\Omega$-classes in its
parameters when a specific relation among the parameters holds.

\subsection{$\Omega$-classes}

For $2g - 2 +n > 0$, consider a non-singular marked curve 
$$
(C; p_1, \ldots, p_n) \in \cM_{g,n}
$$
 and let $\Klog = \omega_C(\sum p_i)$ be its log canonical bundle. Fix a positive integer $r$, and let $1 \leq s \leq r$ and $1 \leq a_1, \ldots, a_n \leq r$ be integers satisfying the equation
\[
a_1 + a_2 + \cdots + a_n \equiv (2g-2+n)s \pmod{r}.
\]
This condition guarantees the existence of a line bundle over $C$ whose $r$th
tensor power is isomorphic to $\Klog^{\otimes s}(-\sum a_i p_i)$. Varying the
underlying curve and the choice of such an $r$th tensor root yields a moduli
space with a natural compactification $\oM_{g;a_1, \ldots, a_n}^{r,s}$ (see e.g. \cite{jarvis}). 
These works also include constructions of the universal curve $\pi :
\overline{\mathcal C}_{g;a_1, \ldots, a_n}^{r,s} \to \oM_{g; a_1, \ldots,
a_n}^{r,s}$ and the universal $r$th root $\cL \to \overline{\mathcal C}_{g;
a_1, \ldots, a_n}^{r,s}$. One can define psi-classes and kappa-classes in complete analogy with the case of moduli spaces of stable curves. 
The Chern characters of the derived pushforward $\mathrm{ch}_k(R^*\pi_*{\mathcal L})$ are given by \cite{chi2}
\begin{align} \label{eq:omegaformula}
\mathrm{ch}_k(r, s; a_1, \ldots, a_n) :={}& \frac{B_{k+1}(s/r)}{(k+1)!} \kappa_k - \sum_{i=1}^n \frac{B_{k+1}(a_i/r)}{(k+1)!} \psi_i^k \notag \\
&+ \frac{r}{2} \sum_{a=0}^{r-1} \frac{B_{k+1}(a/r)}{(k+1)!} {j_a}_* \frac{(\psi')^k + (-1)^{k-1}(\psi'')^k}{\psi'+\psi''}. 
\end{align}
Here, $B_m(x)$ denotes the Bernoulli polynomial, $j_a$ is the boundary morphism that represents the boundary divisor with multiplicity index $a$ at one of the two branches of the corresponding node, and $\psi',\psi''$ are the $\psi$-classes at the two branches of the node.

We define
\[
\Omega_{g,n}^{[x]}(r, s; a_1, \ldots, a_n) := \epsilon_{*} \exp \bigg[ \sum_{k=1}^\infty (-x)^k (k-1)! \, \ch_k(r, s; a_1, \ldots, a_n) \bigg]
\]
for the pushforward of the virtual total Chern class on the moduli space of stable curve via the natural forgetful morphism
\[
\epsilon: \overline{\mathcal{M}}^{r,s}_{g; a_1, \ldots, a_n} \to \overline{\mathcal{M}}_{g,n},
\]
which forgets the line bundle, otherwise known as the spin structure. Moreover $\Omega$ is enriched with an extra parameter $x$ which tracks the cohomological degree of the class.

\subsection{Goulden-Jackson-Vakil ELSV formula}
We now recall the polynomiality result of Goulden-Jackson-Vakil \cite{GJV} for one-part Hurwitz numbers, their conjectural ELSV formula, and the solution found in \cite{DL}.

\noindent
Let $g \geq 0$ and let $x_i$ be positive integers. Let us consider the one-part simple Hurwitz numbers
$$
h_{g; x_1, \ldots, x_n}^{\textnormal{one-part}},
$$
that is, Hurwitz coverings of the Riemann sphere of genus $g$ with ramification profile $\textbf{x}$ over zero, total ramification 
$
(d) = \left(\sum_i^n x_i\right) = \left(e_1(\textbf{x})\right)
$ 
over infinity, and only $2g - 1 + n$ further ramifications elsewhere, all of which are required to be simple. Let $\SSS(x)$ be the the function
\begin{align*}
\SSS(x) &= \frac{\sinh(x/2)}{(x/2)} = \sum_{k=0}^{\infty} \frac{x^{2k}}{2^{2k}(2k + 1)!} = 1 + \frac{x^2}{24} + \frac{x^4}{1920} + O(x^6).
\end{align*}
The inverse function reads
\begin{align*}
\frac{1}{\SSS(x)} &= \frac{(x/2)}{\sinh(x/2)} = 1 + \sum_{k=0}^{\infty} \frac{(2^{1 - 2k} - 1)B_{2k} x^{2k}}{(2k)!} = 1 - \frac{x^2}{24} + \frac{7x^4}{5760} + O(x^6).
\end{align*}

\noindent
Goulden, Jackson and Vakil proved that one-part double Hurwitz numbers are represented by a polynomial in the $x_i$ of degree sharply in the range $[2g - 2 +n, 4g - 2 + n]$.
\begin{theorem}[\cite{GJV}] Simple one-part Hurwitz numbers are polynomial in the ramification profile, i.e. there exist a polynomial $h_{g}^{\textnormal{one-part}}(\textbf{x})$ such that for positive integer values one has $h_{g}^{\textnormal{one-part}}(\textbf{x}) = h_{g, \textbf{x}}^{\textnormal{one-part}}$. More precisely, this polynomial is given by
\begin{equation}\label{eq:gen1}
h_{g}^{\textnormal{one-part}}(\textbf{x}) = e_1(\textbf{x})^{2g-2+n} \, P_{g,n}(x_1, \ldots, x_n),
\end{equation}
for the degree $2g$ polynomial
\begin{align}\label{eq:gen2}
P_{g,n}(x_1, \ldots, x_n) = [t^{2g}]. \frac{\prod_{i=1}^n \SSS(t x_i)}{\SSS(t)} .
\end{align} 
\end{theorem}

It is clear that $P_{g,n}$ is a polynomial of degree $g$ in the variables $x_1^2, \dots, x_n^2$. The result of the theorem above can nowadays be reproduced by a one-line computation using more modern techniques such as semi-infinite wedge formalism, see for instance \cite{johnson}. In the same paper Goulden-Jackson-Vakil have conjectured the existence of an ELSV-type formula for the numbers $h_{g; \textbf{x}}^{\textnormal{one-part}}$. One such ELSV formula was found in \cite{DL}:
\begin{equation}\label{eq:onepartELSV}
h_{g}^{\textnormal{one-part}}(\textbf{x}) = e_1^{2g - 2 + n} \II{g}{n} e_1 \frac{\Omega_{g,n}^{[1]}(e_1 ,e_1 ; e_1^- )}{\prod_{i=1}(1 - \frac{x_i}{e_1} \psi_i)} 
=
e_1^{1 - g} \II{g}{n} \frac{ e_1 \Omega^{[e_1]}_{g,n}(e_1 ,e_1 ; e_1^- )}{\prod_{i=1}(1 - x_i \psi_i)},
\end{equation}
where 
$
e_1^-(x_1, \dots, x_n) := (e_1(\textbf{x}) - x_1, \dots, e_1(\textbf{x}) - x_n).
$
To pass from the first to the second expression simply multiply and divide by $e_1^{3g - 3 +n}$ and pair powers of $e_1$ to cohomological degree. Putting the two results together we have:
\begin{equation}
\II{g}{n} \frac{ e_1 \Omega^{[e_1]}_{g,n}(e_1 ,e_1 ; e_1^- )}{\prod_{i=1}(1 - x_i \psi_i)} = e_1^{3g - 3 +n}\cdot P_{g,n}(x_1, \dots, x_n)
\end{equation}

Let $\Omega_{d,g,n}$ be the cohomological complex degree $d$ part of the $Omega$-class $e_1\Omega_{g,n}^{[1]}(e_1,e_1;e_1^{-})$. It is well-known that
\begin{equation}
e_1\Omega_{0,g,n} = e_{1}^{2g} \cdot \mathbf{1}_{g,n},
\end{equation}
where $\mathbf{1}_{g,n}$ is the fundamental class. Therefore we can extract $A_{g,n}$ from this expression as
\begin{equation}
A_{g,n} 
= e_1^{g - 3 + n}[t^{2g}].  \frac{\prod_{i=1}^n \SSS(t x_i)}{\SSS(t)} \;-\; \sum_{d=1}^{3g - 3 + n} e_1^{d-2g} \II{g}{n} \frac{e_1\Omega_{d,g,n}}{\prod_{i=1}(1 - x_i \psi_i)}
\end{equation}
By analysing the degree in the $\mu_i$ in the expression above, we see that, since of course $A_{g,n}$ is homogeneous of degree $3g - 3 +n$, the subleading terms on the right-hand-side should cancel each other. On the other hand, it is clear that taking the leading coefficient in the $\mu_i$ of the first term of the RHS simply amounts to removing the $\SSS(t)$ from the denominator. Let $\Omega_{g,n,d}^{\textup{top}}$ indicate the polynomial terms in $\Omega_{g,n,d}$ of degree $2g$ in the $\mu_i$. Then we have
\begin{equation}
A_{g,n} 
= e_1^{g - 3 + n}[t^{2g}]. \prod_{i=1}^n \SSS(t x_i) \;-\; \sum_{d=1}^{3g - 3 + n} e_1^{d-2g} \II{g}{n} \frac{e_1\Omega_{d,g,n}^{\textup{top}}}{\prod_{i=1}(1 - x_i \psi_i)}.
\end{equation}

It is obvious that the first term in the RHS satisfies the form of the conjecture \ref{conj:main}: as the term contains $e_1^{g - 3 +n}$ times a polynomial of degree $2g$, the most factors $e_{\lambda_i}$ with $\lambda_i > 1$ is clearly $g$ \footnote{In fact the condition $\textup{top}$ on $\Omega_{d,g,n}$ is quite restrictive: as can be seen for instance in \cite[Corollary 4]{JPPZ} the class $\Omega_{d,g,n}$ has an explicit expression in terms of stable graphs decorated with $\psi$-classes, $\kappa-$classes, and nodes which are either separable or not. Each decoration involves coefficients given by Bernoulli polynomials as above. It is discussed in \cite{QuantumWK} that a lot of the decorations arising from non-separable nodes produce subleading terms in the $\mu_i$, therefore the sum over stable graphs contains mostly sums over stable trees.}. Therefore we immediately obtain a restatement of our main conjecture in terms of $\Omega$-classes:

\begin{proposition} The following two statements are equivalent:
\begin{itemize}
\item[$i).$] Conjecture \ref{conj:main}: $A_{g,n}$ expanded in the $e_i$ basis has at most $g$ factors $e_{\lambda_i}$ with $\lambda_i > 1$.
\item[$ii).$] The expression 
$$
\sum_{d=1}^{3g - 3 + n} e_1^{d-2g} \II_{g,n} \frac{\Omega_{d,g,n}^{\textup{top}}}{\prod_{i=1}(1 - x_i \psi_i)} 
$$
expanded in the $e_i$ basis has at most $g$ factors $e_{\lambda_i}$ with $\lambda_i > 1$.  
\end{itemize}
\end{proposition}
The fact that this particular linear combination of integrals of $Omega$-classes is a polynomial in the variables $x_i$ is per sé quite non-trivial, given that the polynomiality of $Omega$-classes with respect to their parameters in general fails (although this followed already before this work, from the combination of \cite{GJV} and \cite{DL}). $Omega$-classes have a precise expression in terms of decorated stable graphs: the statement above implies that after integrating over all decorated strata represented by these graphs the polynomiality in the parameters features the same type of vanishing as the integrals of pure psi classes, and that this is so for all pairs of stable $(g,n)$. Again, it would be desirable to find a geometric explanation of this phenomenon in terms of the properties of $Omega$-classes.

\section{Application: Intersection numbers as a $g$-dimensional integrals}
\label{sec:gdimintegral}

This section is devoted to integral representations of the amplitudes $A_{g,n}$, as well as to the amplitudes of more involved cohomological classes \textemdash \; the Weil-Petersson polynomials. We remark that the obtained expressions rely on the main conjecture. However, dropping the assumption that conjecture \ref{conj:main} holds true still gives rise to finite integral representations, althought the number of integrations needed gets much messier.

We start from the observation that the Kontsevitch-Witten intersection numbers
can be recovered from the amplitudes as a 
residue taken simultaneously in \( n \) variables:
\begin{equation}
	\cor{ \tau_{d_{1}} \ldots \tau_{d_{n}}}_{g} 
	=
	\underset{\mathbf{x} = 0}{\operatorname{Res}}  
	\prod_{i = 1}^{n} 
	\frac{dx_{i}}{x_{i}^{d_{i} +1}}
	A_{g,n}(\textbf{x}).
	\label{eq:psifromAgn}
\end{equation}
We are going to analyse and refine in several directions this expression, turning it in particular into a $(g+1)$-dimensional integral. We are going to recall first some background on symmetric functions.

\subsection{Monomial symmetric functions}

We shall use an intermediate basis of symmetric functions.
We work with partitions with
non zero entries, so that when
\( n > \ell(\lambda) \), we use the
convention that
\( \lambda_{\ell(\lambda) + 1} ,\ldots, \lambda_{n} = 0 \). 
The monomial symmetric polynomials of a partition
\( \lambda \) in the variables \( \mathbf{x} = (x_{1} ,\ldots, x_{n} ) \) 
is the sum of all monomials 
\( x_{1}^{\nu_{1}} \ldots x_{n}^{\nu_{n}} \)
where \( \nu \) are
\textit{distinct} reorderings of \( \lambda \), this is equivalent to,
\begin{equation}
	m_{\lambda}(x_{1} ,\ldots, x_{n}) = 
	\frac{1}{\zz_{\lambda}(n-\ell(\lambda))!}
	\sum_{\sigma \in \mathfrak{S}_{n}} 
	x_{1}^{\lambda_{\sigma(1)}} \ldots x_{n}
	^{\lambda_{\sigma(n)}},
\end{equation}
where \( \zz_{\lambda} = \prod_{k \geq 1} mult_{k}(\lambda)! \) 
is the product of factorials of multiplicities
of the entries of \( \lambda \) (not to be confused with the stabiliser of the conjugacy class, which generally gets furthermore multiplied by all the individual parts). For instance $\zz_{(5,5,2,2,2)} = 2! 3! = 24$.

Some examples of $m_{\lambda}$ are \( m_{(2,1,1)}(x_{1}, x_{2} , x_{3}) = 
x_{1}^2 x_{2} x_{3} + x_{1} x_{2}^2 x_{3} 
+ x_{1} x_{2} x_{3}^2\) and 
\( m_{(1)}(x_{1}, x_{2} , x_{3}) = e_{1}(x_{1} , x_{2} , x_{3}) \).
The monomial symmetric functions have an orthonormal product that can
be represented by a residue, letting \( \mathbf{v} = ( v_{1} ,\ldots, v_{n}) \),
\begin{equation}
	\underset{\mathbf{v} = 0}{\operatorname{Res}} \prod_{i=1}^{n} 
	\frac{dv_{i}}{v_{i}}
	m_{\lambda}(\mathbf{v}) 
	m_{\mu}(\mathbf{1/v}) 
	= \delta_{\mu , \lambda} 
	\frac{n!}{\zz_{\lambda}(n-\ell(\lambda))!} .
\end{equation}
Moreover the monomial symmetric and elementary
symmetric basis satisfy
\begin{equation}
	\sum_{\substack{\lambda \\ \ell(\lambda) \leq n \\
	\lambda_{i} \leq m}} 
	e_{\lambda}(p_{1} ,\ldots, p_{m})
	m_{\lambda}(v_{1} ,\ldots, v_{n}) 
	=
	\prod_{i = 1}^{n}  
	\prod_{j = 1}^{m} 
	( 1+ p_{j} v_{i}  ),
\end{equation}
For a fixed partition \( \lambda \) and for \( n \geq \ell(\lambda) \),
we can decompose an elementary symmetric functions as
\begin{align}
	e_{\lambda}(\mathbf{p})
	&=
	\underset{\mathbf{v} = 0}{\operatorname{Res}} \prod_{i = 1}^{n} 
	\frac{dv_{i}}{v_{i}}
	\prod_{j=1}^{m} (1+p_{j} v_{i}  )
	\frac{\zz_{\lambda}(n-\ell(\lambda))!}{n!}
	m_{\lambda}(\mathbf{1/v}).
	\\
	&=
	\underset{\mathbf{v} = 0}{\operatorname{Res}} \prod_{i = 1}^{n} 
	\frac{dv_{i}}{v_{i}}
	\prod_{j=1}^{m} (1+p_{j} v_{i}  )
	\frac{1}{n!}
	\tilde{m}_{\lambda}(\mathbf{1/v}).
\end{align}
Where we used the augmented monomial symmetric polynomials for conciseness
(summing over all reorderings, possibly not distinct)
\begin{equation}
	\tilde{m}_{\lambda}(x_{1} ,\ldots, x_{n})
	=
	\sum_{\sigma \in \mathfrak{S}_{n}} 
	x_{1}^{\lambda_{\sigma(1)}} \ldots x_{n}
	^{\lambda_{\sigma(n)}}.
\end{equation}
In order to deal with the factors of \( e_{1} \), we will use the identity
\begin{equation}
	e_{1}^{k}(p_{1} ,\ldots, p_{m})
	=
	k!\,
	\underset{u = 0}{\operatorname{Res}}\, 
	\frac{du}{u^{k+1}} 
	\prod_{j=1}^{m} e^{u p_{j} }.
\end{equation}
\subsection{$g-$dimensional integrals}

For an arbitrary set of times \( t_{k} \), \( k \geq 2 \),
we consider the correlators
\begin{equation}
	\cor{e^{ \sum_{k \geq 2} t_{k-1} \tau_{k}} \prod_{i = 1}^{n} \tau_{d_{i}}}_{g}.
\end{equation}
\noindent
We denote the corresponding amplitudes by
\begin{equation}
	A_{g,n}^{\Omega_{t}} = \sum_{d_{1} ,\ldots, d_{n}} \cor{ e^{\sum_{k \geq 2} t_{k-1} \tau_{k}} \prod_{i = 1}^{n} \tau_{d_{i}} }_{g}
\prod_{i} x_{i}^{d_{i}}.
\end{equation}
Let us comment on the lower bound $k \geq 2$ of the sum in the exponential: having \( \tau_{0} \) 
in the exponential would make the whole expression ill-defined (as the number of marked points is allowed to increase without increasing the cohomological degree, therefore giving rise to an infinite sum for each monomial in the $x_i$);
whereas the \( \tau_{1} \) in the exponential can be 
simplified by means of the dilaton equation:
	\begin{align}
		&\cor{e^{\sum_{k \geq 1} t_{k-1} \tau_{k}}
	\prod_{i = 1}^{n} \tau_{d_{i}} 
		} =
	\tilde{t}_{0}^{2g-2+n} 
	\cor{e^{\sum_{k \geq 2} \tilde{t}_{k-1} \tau_{k}}
	\prod_{i = 1}^{n} \tau_{d_{i}} 
	}_{g} ,
	\end{align}
	under the substitutions
	\( (1-t_{0} )^{-1} \to \tilde{t}_{0} \) and \( t_{i}(1-t_{0})^{-1} \to \tilde{t}_{i}$, for $i \geq 1\). Let \( d_{g,n} = 3g-3+n \) be the dimension of the moduli space of curves.
We compute:
\begingroup
	\allowdisplaybreaks
\begin{align}
	A_{g,n}^{\Omega_{t}} &= \sum_{d_{1} ,\ldots, d_{n}} 
	\cor{
	e^{\sum_{k \geq 2} t_{k-1} \tau_{k}}
	\prod_{i = 1}^{n} \tau_{d_{i}} 
}_{g}
\prod_{i=1}^{n} x_{i}^{d_{i}}
\\
&= \sum_{\ell \geq 0} \frac{1}{\ell !} 
\sum_{\mu_1,\dots,\mu_\ell}
\sum_{d_{1} ,\ldots, d_{n}} 
t_{\mu_1-1} \dots t_{\mu_\ell-1}
<\tau_{\mu \, \sqcup\, d}>_g 
\prod_{i=1}^{n} x_{i}^{d_{i}}
\\
&= \sum_{\ell \geq 0} \frac{1}{\ell !} 
\sum_{\mu_{1} ,\ldots, \mu_{\ell} } 
{t_{\mu_1-1}} \dots t_{\mu_\ell-1}
\underset{\mathbf{p} = 0}{\operatorname{Res}} 
\prod_{j = 1}^{\ell} 
\frac{dp_j}{p_j^{\mu_j+1}} 
{A}_{g,\ell+n}(\mathbf{p},\mathbf{x})  \\
&= \sum_{\ell \geq 0} \frac{1}{\ell !}  
\, \underset{\mathbf{p} = 0}{\operatorname{Res}} \,
\prod_{j = 1}^{\ell}  
\frac{dp_j}{p_j} 
(\sum_{\mu \geq 2} t_{\mu-1} p_j^{-\mu} )  
{A}_{g,\ell+n}(\mathbf{p},\mathbf{x})  \\
&
\begin{aligned}
&=\sum_{\ell \geq 0} \frac{1}{\ell !}  
\sum_{\substack{|\lambda| \leq d_{g,n}+\ell \\ \lambda_{i} \geq 2 \\ l(\lambda) \leq g}}
C_g(\lambda)  
\, \underset{\mathbf{p} = 0}{\operatorname{Res}} \,
\prod_{j = 1}^{\ell}  
\frac{dp_j}{p_j}  
(\sum_{ \mu \geq 2 } t_{\mu-1} p_j^{-\mu} )   
e_\lambda(\mathbf{p},\mathbf{x}) e_1(\mathbf{p},\mathbf{x})^{d_{g,n}+\ell-|\lambda|} 
\end{aligned}
\\
&\begin{aligned}
= 
\sum_{\substack{\lambda_{i} \geq 2 \\ l(\lambda) \leq g}}
C_g(\lambda)  \sum_{\ell \geq |\lambda|-d_{g,n} }   
\underset{\mathbf{v} = 0}{\operatorname{Res}} \,
&\prod_{m=1}^{g}  
\frac{dv_m}{v_m} 
\frac{1}{g!} 
\tilde{m}_\lambda(\mathbf{1/v}) 
\frac{(d_{g,n}+\ell-|\lambda|)!}{\ell !} 
\\
\underset{u = 0}{\operatorname{Res}} \,
\frac{du}{u^{d_{g,n}+\ell-|\lambda|+1}} 
\underset{\mathbf{p} = 0}{\operatorname{Res}} \,
&\prod_{j=1}^{\ell} 
\frac{dp_j}{p_j} 
(\sum_{\mu \geq 2} t_{\mu-1} p_j^{-\mu} )  
e^{u p_j} \prod_{m = 1}^{g}  (1+v_mp_j)
\\
&\prod_{i = 1}^{n} e^{u x_{i}} 
\prod_{m=1}^{g} (1 + v_{m} x_{i} ).
\end{aligned}
\label{eq:laststep}
\end{align}
\endgroup
Now the \( p_{i}\) are independent integration variables,
which we can relabel \( p \).
Let us evaluate the residue in \( p \) and define
\begin{align}
	F(t;u,\mathbf{v})
&\coloneqq u- \underset{p = 0}{\operatorname{Res}} \, \frac{dp}{p}  
(\sum_{\mu \geq 2}
t_{\mu-1}
p^{-\mu}) e^{u p} \prod_{m=1}^{g} (1+v_mp)  \\
&= u-\sum_{\substack{k,j \geq 0 \\ j + k \geq 2}} 
t_{k+j-1}\ 
\frac{u^j}{j!} \ e_k(\mathbf{v}).
\label{def:Ftuv}
\end{align}
We now want to evaluate the sum over \( \ell \) by means of the following series identity:
\begin{equation}
	\sum_{\ell \geq -s} \frac{(\ell+s)!}{\ell!} x^{\ell} =
	\begin{cases}
		\displaystyle s! \, (1-x)^{-(s+1)} & s \geq 0\\
		\displaystyle \frac{(-1)^{s}}{(-s-1)!} (1-x)^{-(s+1)} 
		\log(1-x) + h_{s}(x) & s < 0
	\end{cases}
	\label{eq:newton}
\end{equation}
where \( h_{s}(x) \) is a polynomial of degree \( -s-1 \). 
The radius of convergence is \( |x| < 1 \).
There are two distinct cases:
\\[10pt]
$\bullet$ If $|\lambda|\leq d_{g,n}$:
\begin{equation}
	\sum_\ell u^{-\ell} \ \frac{(d_{g,n}+\ell-|\lambda|)!}{\ell !}  (u-F(t;u,\mathbf{v}))^\ell
	= (d_{g,n}-|\lambda|)! \ \left(\frac{1}{u}F(t;u,\mathbf{v})\right)^{|\lambda|-d_{g,n}-1}
	\label{eq:series1}
\end{equation}
$\bullet$ If $|\lambda|> d_{g,n}$:
\begin{multline}
	\sum_\ell u^{-\ell} \ \frac{(d_{g,n}+\ell-|\lambda|)!}{\ell !}  (u-F(t;u,\mathbf{v}))^\ell\\
	= \frac{(-1)^{|\lambda| - d_{g,n}} }{(|\lambda|-d_{g,n}-1)!} \ \left(\frac{1}{u}F(t;u,\mathbf{v})\right)^{|\lambda|-d_{g,n}-1}
	\ln{\left(\frac{1}{u}F(t;u,\mathbf{v})\right)}
	+ h_{d_{g,n} - |\lambda|}\left(\frac{F(t; u,\mathbf{v})}{u} \right)
	\label{eq:series2}
\end{multline}
One can check that the contribution from the polynomial \( h \) vanishes under the \( u \)-residue when the $v_i$ are small enough, hence one can 
safely discard that term from the computation.
To sum this series, we need to ensure that it is 
converging everywhere on some contours corresponding to the residue integration
in the \( u\) and \( \mathbf{v}\) planes.
Initially centred around \( 0 \), these contours might need to be
deformed in order to satisfy the constraint
$$ \Big{|}1-F(t; u,\mathbf{v})/u \Big{|} < 1. $$
To this end, let us write \( F \) as $F(t; u,\mathbf{v}) = \sum_{k \geq 0} F_{k}(t; u) e_{k}(\mathbf{v})$. Notice that the sum over \( k \) is bounded by the amount of the $v$-variables, which is \( g \). By the triangular identity we have that
\begin{equation}
	\Big{|} 1 - \frac{F(t; u, \mathbf{v})}{u} \Big{|} \leq  \Big{|} 1 - \frac{F_{0}(t; u)}{u} \Big{|} 
	+ \Big{|} \sum_{k = 1}^{g} \frac{F_{k}(t; u)}{u} e_{k}(\mathbf{v}) \Big{|}.
\end{equation}
Clearly, if we keep the \( v_{i} \) small,
the contribution from \( k \geq 1 \) is also small.
So we keep the contours in
the \( \mathbf{v} \) as a residue around zero, and we are 
left with the constraint
\begin{equation}
	\Big{|} 1 - \frac{F_{0}(t; u)}{u} \Big{|} < 1, \qquad \qquad \text{or equivalently} \qquad \qquad \Big{|} \sum_{j=2}t_{j-1}\frac{u^j}{j!} \Big{|} < 1.
	\label{eq:cond}
\end{equation}
Depending on the function \( F(u) \), we might have to 
deform the residue contour at \( u = 0 \) to 
a contour \( \gamma \) in
the region where \eqref{eq:cond} is satisfied,
homotopic with respect to the poles of 
\( F(t ; u,\mathbf{v}) \) in \( u \)
(equivalently, the poles of all \( F_{k}(u) \), as the $e_k$ are linearly independent).

Let us denote by \( \mathfrak{ \hat{u}}(F) \) the set of zeros of \( F(t;u,\mathbf{v}) \) in the region defined by \(| 1 - \frac{F_{0}(u)}{u} | < 1\): the deformed integral along \( \gamma \) picks up contributions labeled by \( \mathfrak{ \hat{u}}(F) \) in two distinct ways, depending on whether the terms are meromorphic or logarithmic. The meromorphic terms of \eqref{eq:series1} yield residues at \( \hat{u} \) for
each \( \hat{u} \in \mathfrak{ \hat{u} }(F) \). The logarithmic terms of \eqref{eq:series2} are considered with the cut in the $u$-variable
on the line segment $[0,\hat{u}]$. If \(\gamma\) encircles this cut, it can be squeezed arbitrarily close to it: it can therefore be considered as the integral over the cut in one direction plus the integral over the cut in the opposite direction in which the argument gets raised by \( 2 \pi i \). After cancellation, the result equals \( 2 \pi i \) times the remaining factors of the integrand, integrated along the cut in the opposite direction. Therefore from \eqref{eq:laststep} we then obtain that $A_{g,n}^{\Omega_t}$ equals
\begin{align}
&
\sum_{\substack{ |\lambda| \leq d_{g,n} \\ \lambda_{i} \geq 2\\ l(\lambda) \leq g }} 
(d_{g,n}-|\lambda|)! \ 
\frac{C_g(\lambda)   }{g!} 
\, \underset{\mathbf{v} = 0}{\operatorname{Res}} \,  
\prod_{m=1}^{g}  \frac{dv_m}{v_m} 
\tilde{m}_\lambda(\mathbf{1/v}) \\
&\hspace{4cm}
\sum_{ \hat{u} \in \mathfrak{ \hat{u} }(F)} 
\underset{u = \hat{u}}{\operatorname{Res}} 
F(t;u,\mathbf{v})^{|\lambda|-d_{g,n}-1}
\prod_{i,m} e^{u x_{i}} (1 + v_{m} x_{i} )
\\
+&\sum_{\substack{ |\lambda| > d_{g,n} \\ \lambda_{i} \geq 2\\ \ell(\lambda) \leq g }} 
\frac{ (-1)^{|\lambda| - d_{g,n}}}{(|\lambda|-d_{g,n}-1)!}   
\frac{C_g(\lambda)}{g!} 
\, \underset{\mathbf{v} = 0}{\operatorname{Res}} \,  
\prod_{m=1}^{g}  \frac{dv_m}{v_m} 
\tilde{m}_\lambda(\mathbf{1/v})\\
&\hspace{4cm}
\sum_{ \hat{u} \in \mathfrak{ \hat{u} }(F)} 
\int_{0}^{ \hat{u}} du
\, F(t;u,\mathbf{v})^{|\lambda|-d_{g,n}-1}
\prod_{i,m} e^{u x_{i}} 
(1 + v_{m} x_{i} )
\end{align} 

\noindent
	When \( \mathfrak{\hat{u}}(F) = \{0\} \), 
	and in particular when 
	\( |1 - F_{0}(u)/u| < 1 \) is verified 
	in a neighbourhood of the origin,
	the line integrals do not contribute, so the last two lines can be discarded.
Let us indicate both operators \( \underset{u = \hat{u}}{\operatorname{Res}} = \frac{1}{2 i \pi} \oint_{ \hat{u}} \) and
\( \int_{0}^{ \hat{u}} \) by the same symbol \( \sqint_{ \hat{u}}\): 
if the integrand is regular in \( F \) it is a line integral on
the segment \( [0, \hat{u} ] \), and if it is meromorphic and singular
it is a residue around \( \hat{u} \).
Finally, we denote
\( \sqint = \sum_{u \in \hat{u}} \sqint_{ \hat{u}} \).
\begin{proposition}
	\label{lem:gdim}
	Assuming conjecture \ref{conj:main}, the times dependent intersection numbers can be written as
	the \( g+1 \) dimensional integral
\begin{multline}
	\label{eq:gdim}
\sum_{d_{1} ,\ldots, d_{n}} 
	\cor{
	e^{\sum_{k \geq 2} t_{k-1} \tau_{k}}
	\prod_{i = 1}^{n} \tau_{d_{i}} 
}_{g}
\prod_{i=1}^{n} x_{i}^{d_{i}}
\\
=
\underset{\mathbf{v} = 0}{\operatorname{Res}} \, \prod_{m=1}^{g}  
\frac{dv_m}{v_m} \sqint du \, \frac{B_{g,n}(F(t;u,\mathbf{v}),\mathbf{v})   
}{F(t;u,\mathbf{v})^{d_{g,n} + 1 } }
\prod_{i,m} e^{u x_{i}} 
(1 + v_{m} x_{i} ),
\end{multline}
where
\begin{multline}
	\label{def:Bgn}
	B_{g,n}(\xi,v_1,\dots,v_g)=
\sum_{\substack{ |\lambda| \leq d_{g,n} \\ \lambda_{i} \geq 2\\ l(\lambda) \leq g }} 
(d_{g,n}-|\lambda|)! \ 
\frac{C_{g}(\lambda)}{g!} 
\tilde{m}_\lambda(\mathbf{1/v}) \, \xi^{|\lambda|}
\\
+
\sum_{\substack{ |\lambda| > d_{g,n} \\ \lambda_{i} \geq 2\\ l(\lambda) \leq g }} 
\frac{ (-1)^{|\lambda| - d_{g,n}}}{(|\lambda|-d_{g,n}-1)!}   
\frac{C_{g}(\lambda)}{g!} 
\tilde{m}_\lambda(\mathbf{1/v}) \, \xi^{|\lambda|},
\end{multline}
and
$
	F(t;u,\mathbf{v})
= u-\sum_{k,j} 
t_{k+j-1}\ 
\frac{u^j}{j!} \ e_k(\mathbf{v}).
$

\end{proposition}

\noindent
In order to extract the coefficient
\(
	\cor{e^{ \sum_{k \geq 2} t_{k-1} \tau_{k}} \prod_{i = 1}^{n} \tau_{d_{j}}}_{g}\)
	we take the residue in the variables \( x_{i} \), which
	can be computed explicitly:
	\begingroup
	\allowdisplaybreaks
	\begin{align*}
		\underset{\mathbf{x} = 0}{\operatorname{Res}}  \prod_{i=1}^{n} \frac{dx_{i}}{x_{i}^{d_{i} +1}}
		e^{u x_{i}} \prod_{m=1}^{g} (1+v_{m} x_{i} )
			&=
			\prod_{i=1}^{n} \frac{dx_{i}}{x_{i}^{d_{i} +1}}
			\sum_{k \geq 0} \frac{u^{k} x_{i}^{k}}{k!} 
			\sum_{r = 0}^{g} e_{r}(\mathbf{v}) x_{i}^{r}
			\\
			&=
			\prod_{i=1}^{n}
			\sum_{r = 0}^{d_{i}}
			\frac{u^{d_{i}-r}}{(d_{i} - r)!} e_{r}(\mathbf{v})
			\\
			&=
			u^{d_{g,n}} 
			\prod_{i=1}^{n}
			\sum_{r = 0}^{d_{i}}
			\frac{e_{r}(\mathbf{v})}{u^{r}}
			\frac{1}{(d_{i} - r)!}.
	\end{align*}
	\endgroup
	\begin{corollary}
		The times dependent correlators can be computed as the following  \( (g+1) \)-dimensional
		integral:
		\begin{equation}
	\cor{
	e^{\sum_{k \geq 2} t_{k-1} \tau_{k}}
	\prod_{i = 1}^{n} \tau_{d_{i}} 
}_{g}
=
\underset{ \mathbf{v} = 0}{\operatorname{Res}} \,
\prod_{m=1}^{g}  
\frac{dv_m}{v_m} \sqint du \, u^{d_{g,n}}
\, \frac{B_{g,n}(F(t;u,\mathbf{v}),\mathbf{v})   
}{F(t;u,\mathbf{v})^{d_{g,n} + 1 } }
\prod_{i=1}^{n}
\sum_{r = 0}^{d_{i}}
			\frac{e_{r}(\mathbf{v})}{u^{r}}
			\frac{1}{(d_{i} - r)!}.
\end{equation}
\end{corollary}
\begin{remark}
	When \( \mathfrak{\hat{u}}(F) = \{0\} \), 
	and in particular when 
	\( |1 - F_{0}(u)/u| < 1 \) is verified 
	in a neighbourhood of the origin,
	the line integrals in \( \sqint \) 
	do not contribute
	and it is enough to restrict \( B_{g,n} \) to the part that
	yields negative powers of \( F \), that is, to
\begin{equation}
	\label{def:Bgnmin}
	B^{-}_{g,n}(\xi,v_1,\dots,v_g)=
\sum_{\substack{ |\lambda| \leq d_{g,n} \\ \lambda_{i} \geq 2\\ l(\lambda) \leq g }} 
(d_{g,n}-|\lambda|)! \    
\frac{C_g(\lambda)}{g!} 
\tilde{m}_\lambda(\mathbf{1/v}) \xi^{|\lambda|}.
\end{equation}
In that case \( \sqint (\cdot)\, du \) is a simple residue in \( u \) at the origin.
\end{remark}

\begin{corollary} Setting all times \( t_{k} \)'s to zero, we recover the Witten--Kontsevich
intersection numbers of $\psi-$classes. Therefore in terms of the corresponding amplitudes we have:
\begin{align}
	A_{g,n}(\mathbf{x})
	&=
\underset{\mathbf{v} = 0}{\operatorname{Res}}  
\prod_{m=1}^{g}  
\frac{dv_m}{v_m} 
\underset{u = 0}{\operatorname{Res}} \frac{du}{u} \, 
\frac{B^{-}_{g,n}(u,\mathbf{v})}{u^{d_{g,n}+1} }
\prod_{i,m} e^{u x_{i}} 
(1 + v_{m} x_{i} )
\\
&=
[ u^{0} v_{1}^{0} \ldots v_{g}^{0}] 
\,
\frac{B^{-}_{g,n}(u,\mathbf{v})}{u^{d_{g,n}} }
\prod_{i,m} e^{u x_{i}} 
(1 + v_{m} x_{i} )
\end{align}
Equivalently, for the correlators we have:
		\begin{align}
	\cor{\tau_{d_{1}} \ldots \tau_{d_{n}} }_{g}
&=
\underset{\mathbf{v} = 0}{\operatorname{Res}}  
\prod_{m=1}^{g}  
\frac{dv_m}{v_m} 
\underset{u = 0}{\operatorname{Res}} \frac{du}{u} \, 
B^{-}_{g,n}(u,\mathbf{v})   
\prod_{i=1}^{n}
			\sum_{r = 0}^{ d_{i}}
			\frac{e_{r}(\mathbf{v})}{u^{r}}
			\frac{1}{(d_{i} - r)!}
\\
&=
[ u^{0} v_{1}^{0} \ldots v_{g}^{0}] 
\,
B^{-}_{g,n}(u,\mathbf{v})   
\prod_{i=1}^{n}
\sum_{r = 0}^{d_{i}}
			\frac{e_{r}(\mathbf{v})}{u^{r}}
			\frac{1}{(d_{i} - r)!}
\end{align}
\end{corollary}

\begin{remark}
Note that taking the coefficient of the $v_i^0$ restricts the partitions $\lambda$ of $B_{g,n}^{-}$ further to partitions with $\lambda_i \leq n$, as it is impossible to collect a power of $v_i$ higher than $n$ from the product of $n$ elementary symmetric polynomials in the $v_j$. 
\end{remark}

Lemma \ref{lem:gdim} provides a much general formula allowing to specify different times $t_k$ which vary from problem to problem. The key ingredient of the lemma is \( B_{g,n} \), which is itself times independent. As of now we are only able to compute \( B_{g,n} \) for a fixed genus \( g \) (and for all \( n \)) from the data of the amplitude \( A_{g,6g-3} \). The amplitude has to be computed with different methods, such as Virasoro constraints. However, we think that the study of the coefficients \( C_{g}(\lambda) \) beyond string and dilaton equations will provide new ways of computing \( B_{g,n} \) in a more direct way.

\subsection{Explicit functions $B_{g,n}$ in low genera}
The functions $B_{g,n}$ and $B^{-}_{g,n}$ for low genus can be given explicitly as follows.

\subsubsection{Genus zero:}

\begin{gather}
	B_{0,n} = \delta_{n < 3} \frac{(-1)^{n+1}}{(2-n)!} + \delta_{n \geq 3} (n-3)!
	\\
	B^{-}_{0,n} = \delta_{n \geq 3} (n-3)!
	\label{eq:B0nmin}
\end{gather}

\subsubsection{Genus one:}
\begin{equation}
	\begin{aligned}
		B_{1,n}( \xi , v)
	&= n! - \sum_{k = 2}^{n} (n-k)! (k-2)!
	\left(\frac{\xi}{ v} \right)^{k}
	+
	\sum_{k = n+1}^{\infty}
	\frac{(-1)^{n-k+1}}{(k-n+1)!}(k-2)! \left(\frac{\xi}{v} \right)
	\\
	&= n! - \sum_{k = 2}^{n} (n-k)! (k-2)!
	\left(\frac{\xi}{ v} \right)^{k}
	+
	\left(\frac{\xi}{ v} \right)^{n+1}
	\frac{(n-1)!}{( 1 + \xi/v)^{n}}
	\end{aligned}
\end{equation}

\begin{gather}
	B^{-}_{1,n}( \xi , v)
	= n! - \sum_{k = 2}^{n} (n-k)! (k-2)!
	\left(\frac{\xi}{ v} \right)^{k}
\end{gather}

\subsubsection{Genus two:}
\begin{equation}
	\begin{aligned}
	B_{2,n}^{-}(\xi , v_{1} , v_{2}) 
	=& \, (n+3)!
	-
	(n+1)! \left( \frac{1}{v_{1}^{2}} + \frac{1}{ v_{2}^{2}} \right) 
	\xi^{2}
	- \frac{9}{5} n!
	\left( \frac{1}{v_{1}^{3}} + \frac{1}{ v_{2}^{3}} \right) 
	\xi^{3} 
     	\\&-
	\sum_{k = 4}^{n+3} (n-k+3)! 
	\frac{(k-3)!}{60} (k^{3} + 21 k^2 - 70 k + 96)
	\left( \frac{1}{v_{1}^{k}} + \frac{1}{ v_{2}^{k}} \right) 
	\xi^{k} 
	\\&
	+ \frac{9}{5} (n-1)!
	\frac{1}{v_{1}^{2} v_{2}^2}
	\xi^{4}
	+ \frac{9}{5} (n-2)!
	\left( \frac{1}{v_{1}^{3} v_{2}^2} + \frac{1}{ v_{1}^2 v_{2}^{3}} \right) 
	\xi^{5} 
     	\\& + \sum_{k = 4}^{n+1} (n-k+1)! 
	\frac{(k-1)!}{20} (k+16)
	\left( \frac{1}{v_{1}^{k} v_{2}^2} + \frac{1}{ v_{1}^2v_{2}^{k}} \right) 
	\xi^{k+2} 
     	\\&- \sum_{k = 3}^{n} (n-k)! \frac{k!}{20} 
	\left( \frac{1}{v_{1}^{k}v_{2}^{3}} + \frac{1}{ v_{1}^{3} v_{2}^{k}} \right) 
	\xi^{k+3} 
\end{aligned}
\end{equation}

\subsection{Application to Weil-Petersson volumes}

In the previous section we developed a times dependent formula for applications to different enumerative geometric problems. In this section we employ this formula specialising it to the suitabel times $t_k$ providing the Weil-Petersson volumes. 

It can be interesting to notice how in this case Bessel functions naturally arise from the function $F$. The fact that Bessel functions must be involved in the amplitudes of the Weil-Petersson volumes has been known in the literature for a fairly long time, as explained below, and formulae in low genera were made precise. Nevertheless, we achieve a concrete formula for their involvement in all genera which, to the best of our knowledge, is new.

For a vector of positive real numbers \( \mathbf{L} = ( L_{1} ,\ldots, L_{n}) \),
let \( \mathcal{M}_{g,n}(\mathbf{L}) \) be the moduli space of 
genus \( g \) hyperbolic surfaces with \( n \) hyperbolic boundaries of length
prescribed by \( \mathbf{L} \). Weil-Petersson volumes are defined as the integral of the Weil-Petersson metric
form \( \omega \) on \( \mathcal{M}_{g,n}(\mathbf{L}) \):
\begin{equation}
	V_{g,n}(\mathbf{L}) = \frac{1}{(3g-3+n)!} 
	\int_{\mathcal{M}_{g,n}(\mathbf{L})} \omega^{3g-3+n}.
	\label{eq:VMgnL}
\end{equation}
The volume form extends to a closed form on \( \overline{\mathcal{M}}_{g,n} \) 
and defines a cohomology class \( [ \omega ] \in H^{2}( \overline{\mathcal{M}}_{g,n} , \mathbb{R})  \). It is also known \cite{Wolpert} that
\( [\omega] = 2 \pi^2 \kappa_{1} \) 
where \( \kappa_{1} \) is the
first Mumford class on \( \overline{\mathcal{M}}_{g,n} \).
In \cite{Mirzakhani} Mirzakhani shows that \eqref{eq:VMgnL}
reduces to an integral on the moduli space \( \overline{\mathcal{M}}_{g,n} \) of stable curves by
\begin{equation}
	V_{g,n}(\mathbf{L}) 
	= 
	\sum_{\alpha_{1} ,\ldots, \alpha_{n}} 
	\cor{ e^{2 \pi^2 \kappa_{1}} \prod_{i = 1}^{n} \tau_{d_{i}}}_{g} 
	\prod_{i} \frac{L^{2 d_{i}}}{2^{d_{i}} d_{i}!}.
\end{equation}
By projection formula, the integrals of monomials of $\psi-$ and $\kappa-$classes can be computed as a combination of integrals of monomials purely in terms of $\psi-$classes. In particular, monomials in $\psi-$classes and powers of \( \kappa_{1} \) can be expressed as combinations of correlators in terms of \( \tau_{k} \)'s.

\begin{align*}
	\left<e^{2 \pi^2\kappa_{1}}\right>_{g} 
	&=
	\sum_{p} \frac{1}{p!}
	\sum_{\substack{m_{1}, \ldots ,  m_{p}\\ m_{i} > 0} }
	\frac{(-1)^{\sum_{i} m_{i} + p } (2 \pi^2 )^{\sum_{i} m_{i}} }{m_{1} ! \ldots m_{p} !} 
	\cor{\prod_{j = 1}^{p} \tau_{m_{j} + 1}}_{g} = \left< e^{\sum_{k \geq 2} t_{k-1} \tau_{k}}\right>_{g}
\end{align*}
for
\begin{equation}
	t_{k} = - \frac{ (-2 \pi^2)^{k}}{k!}, \qquad k \geq 1.
	\label{eq:mirztimes}
\end{equation}
For the rest of this section \( F(t,u,\mathbf{v}) 
\) is intended with specialized times. Let us compute 
\begin{align}
	F(u,\mathbf{v})
&= 
u - 
\sum_{\substack{j,k \geq 0 \\ j+k \geq 2}} 
{t_{k+j-1}} \ \frac{u^j}{j!} \ e_k(\mathbf{v}),\\
&= 
u - 
\sum_{\substack{j,k \geq 0 \\ (j,k) \neq (0,0) }}
	{t_{k+j-1}} \ \frac{u^j}{j!} \ e_k(\mathbf{v})
+ t_{0} u + t_{0} e_{1}(\mathbf{v}),
\\
&=
- e_{1}(\mathbf{v}) +
\sum_{k \geq 0} 
e_{k}(\mathbf{v}) 
(-2\pi^2)^{k-1} 
\sum_{j \geq 0} 
\frac{(-1)^{j}}{j! (j+k-1)!} (2 \pi^2)^{j} u^{j} .
\end{align}
Taking the convention that \( 1/(-1)! = 0 \).
We can relate this to the Bessel functions,
\begin{gather}
	J_{n}(x) = 
	\sum_{j = 0}^{\infty} \frac{(-1)^{j}}{j! (j+n)!} \left( \frac{x}{2} \right)^{2j+n},
	\\
	I_{n}(x) = 
	\sum_{j = 0}^{\infty} \frac{1}{j! (j+n)!} \left( \frac{x}{2} \right)^{2j+n}.
\end{gather}
We shall also use that \( J_{-n}(x) = (-1)^{n} J_{n}(x) \). We obtain
\begin{equation}
	\label{eq:FMirz}
	F(u,\mathbf{v}) = - e_{1}(\mathbf{v}) 
	+ \sum_{k \geq 0} e_{k}(\mathbf{v}) 
	(-1)^{k+1}
	\left(\frac{\pi \sqrt{2}}{\sqrt{u}}\right)^{k-1}
	J_{k-1}(2 \pi \sqrt{2u}).
\end{equation}
Here the sum over \( k \) is  bounded by \( g \) since 
we conjecture that at most \( g \) variables of 
integration \( v_{1} ,\ldots, v_{g} \) are necessary.
Let us work out the set \( \hat{\mathfrak{u}}(F) \)
to enter into the residue formula.
The condition
\( | 1 - F_{0}(u)/u | < 1 \) is verified in a neighbourhood
of the origin, as \( F_{0}(u)/ u = \frac{1}{\pi \sqrt{2 u}} J_{1}(2 \pi \sqrt{2 u})  \) and \( \lim_{z \to 0} J_{1}(2z)/z = 1 \).
So we can keep the residue contour around \( u = 0 \).
Then, we can check that \( u=0 \) is indeed (the unique) solution
of \( F(u,\mathbf{v}) = 0 \) since
\( J_{0}(0) = 1 \) and \( J_{k}(0) = 0  \) for \( k>1 \).
Therefore \( \hat{\mathfrak{u}}(F) = \{0\} \), and only the terms 
in \( B^{-}_{g,n} \) will contribute to the residue formula
\eqref{eq:gdim}.
The first values of \( F(u, \mathbf{v}) \) for \( g = 0,1,2 \) are
\begin{align}
	F(u)
	&=
	\frac{1}{\pi} \sqrt{\frac{u}{2}} J_{1}(2 \pi \sqrt{2u})
\end{align}
\begin{equation}
	F(u,v_{1}) =
	v_{1} 
	\left(J_{0}(2 \pi \sqrt{2u})-1\right)
	+
	\frac{1}{\pi} \sqrt{\frac{u}{2}}
	J_{1}(2 \pi \sqrt{2u})
\end{equation}
\begin{align}
	F(u,v_{1}, v_{2}) 
	 &= (v_{ 1} + v_{2}) \left( J_{0}(2 \pi \sqrt{2u}) -1 \right) 
	 + \left( 1 + 2 \pi^2 \frac{v_{1}v_{2}}{u} \right) \frac{1}{\pi} 
	 \sqrt{\frac{u}{2}}
	J_{1}(2 \pi \sqrt{2u}) 
\end{align}
The last ingredient we need to relate \( A_{g,n}^{\Omega_{t}} \) to
\( V_{g,n} \) is the change of basis \( x_{i} \leftrightarrow L_{i} \).
Let
\begin{equation}
	\mathcal{L}_{x} f(L) \coloneqq
	\frac{1}{x}
	\int_{0}^{\infty} f(L) e^{- \frac{L^2}{2x}} L dL, 
	\qquad \qquad \text{ so that } \qquad 
	x_{i}^{d_{i}} = \mathcal{L}_{x_{i}} 
	\left(\frac{L^{2 d_{i}}}{2^{d_{i}} d_{i} !}\right).
	\label{eq:xtoL}
\end{equation}
In particular, we have \( V_{g,n}(\mathbf{L}) = \mathcal{L}^{-1} A_{g,n}^{\Omega_{t}} (\mathbf{x})\). We also compute
\begin{align*}
\mathcal{L}^{-1}
\prod_{m} 
e^{u x_{j}} 
(1 + v_{m} x_{j} )
&=
\mathcal{L}^{-1}
\sum_{k,\ell} \frac{u^{\ell}}{\ell!} 
e_{k}(\mathbf{v}) 
x_{i}^{k+\ell}
\\
&=
\sum_{k,\ell} \frac{u^{\ell}}{\ell!} 
e_{k}(\mathbf{v}) 
\frac{L_{i}^{2k+2\ell}}{2^{k+\ell} (k+\ell)!}
\\&=
\sum_{k}
e_{k}(\mathbf{v}) 
\left(\frac{L_{i}}{\sqrt{2u}}\right)^{k}
\sum_{\ell} 
\frac{1}{\ell!(k+\ell)!} 
\left(
	\frac{L_{i}\sqrt{2u}}{2}
\right)^{2\ell+k}
\\&=
\sum_{k}
e_{k}(\mathbf{v}) 
\left(\frac{L_{i}}{\sqrt{2u}}\right)^{k}
I_{k}(L_{i}\sqrt{2u} )
\end{align*}

\noindent
We can now specialize \eqref{eq:gdim} to the case of Weil-Petersson volumes.
\begin{theorem} 
Conjecture \ref{conj:main} implies the following integral representation of the Weil-Petersson amplitudes:
 \begin{equation}
	\label{eq:ConjMirz}
		\begin{aligned}
	&V_{g,n}(\mathbf{L}) =
	\, \underset{\mathbf{v} = 0}{\operatorname{Res}} \,\prod_{m=1}^{g}  
\frac{dv_m}{v_m} 
\, \underset{u = 0}{\operatorname{Res}} \,
du\\&\hspace{2cm}
\frac{B^{-}_{g,n}(F(u,\mathbf{v}),\mathbf{v})   
}{F(u,\mathbf{v})^{3g-2+n} }
\prod_{i = 1}^{n} 
(
\sum_{k} e_{k}(\mathbf{v}) 
\left( \frac{L_{i}}{\sqrt{2u}} \right)^{k} I_{k}(L_{i} \sqrt{2u})
)
\end{aligned}
\end{equation}
\end{theorem}

\begin{remark}
In fact a finer statement holds: each $(g,n)$ for which conjecture \ref{conj:main} is proved implies the corresponding statement in the result above. Therefore formula \eqref{eq:ConjMirz} holds for $g \leq 7$, any $n$; and for $n \leq 3$, any $g$.
\end{remark}

\noindent
We give in the following the explicit formulae for $V_{g,n}$ for genus \( 0 \)
and \( 1 \). These formulae are actually proved, as both the main conjecture and
the formulae for $B_{g,n}$ were proved in low genus.
We applied the change of variable \( 2u \to u  \) in the residue.
\begin{equation*}
	{V}_{0,n}(L_{1} ,\ldots, L_{n})  = 
	\frac{1}{2}
	(n-3)!
	\left(2\pi \right)^{n - 2} 
	\underset{u = 0}{\operatorname{Res}} 
	\frac{du
	}{( \sqrt{u} J_{1}(2 \pi \sqrt{u}))^{n-2}} 
	\prod_{i = 1}^{n} 
	I_{0}(L_{i} \sqrt{u})
\end{equation*}
For \( g = 1 \) we can compute the residue in \( v_{1} \) explicitly.
The result is,
\begin{equation*}
		\begin{aligned}
			V_{1,n}&(L_{1} ,\ldots, L_{n}) 
			=
	\frac{1}{48} (2 \pi)^{n+1}
	\,
	\underset{u = 0}{\operatorname{Res}} \,
du 
\frac{1}{(\sqrt{u} J_{1}(2 \pi \sqrt{u}) )^{n+1}}
\prod_{j = 1}^{n} I_{0}(L_{j} \sqrt{u}) 
\\
&\bigg(n! - 
	\sum_{k=2}^{n}  (n-k)!(k-2)! \sum_{\ell = 0}^{k}
	{{n-\ell}\choose{k-\ell}}
	(\frac{\sqrt{u}}{ \pi } J_{1}(2 \pi \sqrt{u}))^{\ell} e_{\ell}(\mathbf{D}) 
	(1-J_{0}(2 \pi \sqrt{u}))^{k-\ell} 
\bigg)
\end{aligned}
\end{equation*}
\noindent
with 
$$
D_{i}(u , L_{i})  =  \frac{ (L_{i} /2) I_{1}(L_{i} \sqrt{u}) }{ \sqrt{u} I_{0}(L_{i}\sqrt{u})} .
$$
Both are checked to give rise to the right Weil-Petersson volumes. Relations 
between generating series of Weil-Petersson polynomials 
and Bessel functions can be found for instance in \cite{OkuSak} in relation to JT Gravity, in \cite{KMZ_Mirza} via combinatorial techniques, and in \cite{MirZog} 
what concerns the Weil-Petersson volumes and their large genus asymptotics.

\section{Examples of similar behaviour from other coholomogical classes}\label{sec:examples}

This section collects a few examples of cohomological classes or cohomological field theories whose amplitudes (as in the case of integrals of $Omega$-classes) show analogies with respect to the ones of the trivial CohFT, when expanded in the elementary symmetric polynomials.

\subsubsection{Chern classes of the Hodge bundle: amplitudes of $\lambda-$classes}

The algebraic geometry of moduli spaces of stable curves has several surprising properties involving its tautological cohomology classes. For example, the following conjecture was proposed by Faber in \cite{Faber} and then proved by Faber and Pandharipande in \cite{FP}. Let $\lambda_i = c_i(\mathbb{E})$ be the $i$-th Chern class of the Hodge bundle of abelian differentials. One has:
\begin{equation}
\II gn \lambda_g \psi_1^{d_1} \cdots \psi_{n}^{d_n} = \binom{2g - 3 +n}{d_1, \dots, d_n} \II g1 \lambda_g \psi_1^{2g - 2}.
\end{equation}
In fact, the integral for $n=1$ has been computed to be 
$$
b_g := \II g1 \lambda_g \psi_1^{2g - 2} = [t^{2g}]. \frac{t/2}{\sin(t/2)}.
$$
Defining 
$$
A_{g,n}^{\Omega = \lambda_g}(\textbf{x}) :=  \II gn \frac{ \lambda_g}{\prod_{i=1}^n (1 - x_i \psi_i)},
$$
it is straightforward to recast the statement above in terms of the generating series, obtaining the cleanest form
\begin{equation}
A_{g,n}^{\lambda_g} = b_g \cdot e_1^{2g - 3 +n}.
\end{equation}
In this example we see that the amplitudes of the class $\lambda_g$ allow for at most \textit{zero} factors of $e_{\lambda}$ with $\lambda > 1$ in their expansion.



Concerning lower $\lambda-$classes, one can consider and analyse the amplitudes
$$
\tilde{A}_{g,n}^{\lambda_i} = 24^g g! \cdot \II_{g,n} \frac{\lambda_i}{\prod_{i=1}^n(1 - x_i \psi_i)}
$$
in the basis of elementary symmetric polynomials.  What we know already about these amplitudes is that for $i=0$
$$
\tilde{A}_{g,n}^{\lambda_0} =  \tilde{A}_{g,n}
$$
the corresponding amplitude is addressed by the main conjecture and its expansion is expected to have at most $g$ factors $e_j$ with $j >1$. One could wonder whether $\lambda_i$ for arbitrary $i$ contraints the maximum amount of such factors to $g-i$.  Experimentally, we see that this is false.  For instance for $g=3$,  $n=3$, and $i=1$ we computed: 
\begin{align*}
\tilde{A}_{3,3}^{\lambda_1}  &= \frac{21}{5}e_1^8 - \frac{42}{5}e_{2} e_1^6 + \frac{321}{35} e_2^2 e_1^4 - \frac{18}{7} e_2^3e_1^2 - \frac{642}{35}e_3e_1^5 + \frac{912}{35} e_3 e_2 e_1^3 - \frac{72}{7}e_3e_2^2e_1 
\\
 &+ \frac{111}{35} e_3^2 e_1^2 + \frac{66}{35}e_3^2e_2,
 \\
 \tilde{A}_{3,3}^{\lambda_2}  &= \frac{41}{7}e_1^7- \frac{41}{7}e_2e_1^5 + \frac{74}{35}e_2^2e_1^3 - \frac{16}{35}e_2^3e_1 - \frac{353}{35}e_3e_1^4 + \frac{54}{7}e_3e_2e_1^2 - \frac{16}{7}e_3e_2^2 + \frac{64}{35}e_3^2e_1.
\end{align*}
Let us remark two aspects: the first is that partitions with $3$ factors $e_j$ with $j > 1$ appear,  hence there is in general no reduction of the number of factor depending on the index of the $\lambda-$class,  except for the intersection of the $\lambda_g$ where all such factors disappear at once.  It is not so surprising that $\lambda_g$ is the only $\lambda-$class playing a special role,  as it is the top degree of a CohFT and moreover it defines an isomorphism between the top tautological rings of the moduli spaces of stable curves and the moduli space of compact type (see, e.g., \cite{Pixton}). 

The other observation is that in this example the structure of the conjecture is not broken: the partition $(2,2,2,2)$ is in principle allowed for $i=1$,  as the homogeneous degree is $8$ in this case,  but in fact it does not appear.  One could think that this fact is an immediate consequence of the main conjecture,  as all intersection of $\psi-$classes are positive rational numbers and hence there is no chance to get non-trivial cancellations which get spoiled by the intersection of the $\lambda_i-$ classes.  However,  the intersection of the $\psi-$classes in $A_{g,n}$ actually takes place in top degree,  whereas in $A_{g,n}^{\lambda_i}$ the $\psi-$classes monomials are of cohomological degree $3g - 3 + n - i$,  and hence partitions with more than $g$ factors $e_j$ with $j >1$ could in principle appear.

\subsubsection{$Omega$-classes}
The $Omega$-class $\Omega^{[x]}(r,s;a_1, \dots, a_n)$ has been discussed in Section \ref{sec:omega} and it is proved to provide a CohFT with flat unit whenever $0 \leq s \leq r$. This is the case for the parametrisation $\Omega^{[e_1]}_{g,n}(e_1 ,e_1 ; e_1^- )$ used in this work.  We have also observed that
\begin{equation}
\II{g}{n} \frac{ e_1 \Omega^{[e_1]}_{g,n}(e_1 ,e_1 ; e_1^- )}{\prod_{i=1}(1 - x_i \psi_i)} = e_1^{3g - 3 +n}\cdot P_{g,n}(x_1, \dots, x_n),
\end{equation}
where $P_{g,n}$ is a known polynomial of degree $2g$ in the $x_i$ which is in fact a polynomials in the $x_i^2$.  What is however relevant for the discussion here is that the degree $2g$ does not allow for any term of the form $e_{\lambda_i} \cdots e_{\lambda_{g+1}}$ in which each $\lambda_i > 1$ in the expansion of the RHS.  This is conforming to the principle of the main conjecture.  In addition, this case guarantees a minimal power of $e_1$, which the amplitudes $A_{g,n}$ for the trivial CohFT do not have.

\subsubsection{Monotone Hurwitz numbers}

Monotone Hurwitz numbers also enumerate branched coverings of the Riemann sphere with prescribed ramifications over zero by a partition $\mu$, and the remaning $2g - 2 + \ell(\mu) + |\mu|$ simple ramifications satisfy the following monotonic property: label the cover sheets, represent each simple ramification by a transposition $(a_i \; b_i)$ written such that $1 \geq a_i < b_i \geq |\mu|$, then impose the condition $b_i \geq b_{i+1}$ for $i = 1, \dots, 2g - 3 + \ell(\mu) + |\mu|$.

Analogously to simple Hurwitz numbers, monotone Hurwitz numbers are known to satisfy an explicit ELSV-type formula and are computed explicitly in genus one and in genus zero. 

The ELSV formula reads

\begin{equation}\label{eq:ELSVmonotone}
\frac{h^{\leq}_{g, \mu}}{(n + d)!} = \prod_{i=1}^\ell \binom{2\mu_i}{\mu_i} \sum_{d_1, \dots, d_n = 0}\int_{\MMMbar_{g, n}} \!\!\!\!\!\Omega^{\leq}_{g,n} \;\; \psi_1^{d_1} \cdots \psi_\ell^{d_n} \prod_{i=1}^\ell \frac{ (2(d_i + \mu_i) - 1)!!}{(2 \mu_i - 1)!!}
\end{equation}
where the intersecting class is
\begin{equation}
 \Omega_{g,n}^{\leq} := e^{\sum_{p=1} A_p \kappa_p }, \qquad \qquad \exp\left( - \sum_{p=1} A_p U^p \right) = \sum_{k=0} (2k + 1)!! U^k.
\end{equation}
On the other hand, the Goulden--Guay-Paquet--Novak formula for monotone Hurwitz numbers in genus zero reads
\begin{equation}
h^{\leq}_{0, \mu} = \prod_{i=1}^\ell \binom{2\mu_i}{\mu_i} 
\frac{\left(2\left(\sum_i \mu_i \right) + \ell - 3\right)!}{\left(2 \left(\sum_i \mu_i \right)\right)!}
\end{equation}
Putting the two together and setting $x_i = 2\mu_i + 1$ we obtain 
\begin{equation}
\sum_{d_1, \dots, d_n = 0}\int_{\MMMbar_{0, n}} \!\!\!\!\! \Omega^{\leq} \;\; \psi_1^{d_1} \cdots \psi_n^{d_\ell} \prod_{i=1}^n \frac{ (2(d_i-1) + x_i)!!}{(x_i - 2)!!}
= 
\frac{\left(e_1 - 3\right)!}{\left( e_1  - n \right)!}
\end{equation} 
It is clear the the RHS of the formula above is a non homogeneous polynomial in $e_1$ of degree $n-3$. There are no $e_j$ involved in the expression for $j>1$, again agreeing with the general principle of the main conjecture \ref{conj:main}.

\noindent
In genus one, the Goulden--Guay-Paquet--Novak formula reads
\begin{equation*}
h^{\leq}_{1, \mu} = \prod_{i=1}^n \binom{2\mu_i}{\mu_i} \left[
\frac{e_1!}{(e_1 - n)!} - 3\frac{(e_1 - 1)!}{(e_1 - n)!} - \sum_{k=1}^n (k-2)! \frac{(e_1 - k)!}{(e_1 - n)!} e_k
\right]
\end{equation*}

Therefore we obtain:
\begin{multline}
\sum_{d_1, \dots, d_n = 0}\int_{\MMMbar_{1, n}} \!\!\!\!\! \Omega^{\leq} \;\; \psi_1^{d_1} \cdots \psi_n^{d_n} \prod_{i=1}^n \frac{(2(d_i-1) + x_i)!!}{(x_i - 2)!!}= \\
=
\frac{e_1!}{(e_1 - n)!} - 3\frac{(e_1 - 1)!}{(e_1 - n)!} - \sum_{k=1}^n (k-2)! \frac{(e_1 - k)!}{(e_1 - n)!} e_k.
\end{multline}
By considering the free energies 
\begin{equation}
A_{g,n}^{\Omega^{\leq}}(\textbf{x}) := \sum_{d_1, \dots, d_n = 1}\int_{\MMMbar_{g, n}} \!\!\!\!\! \Omega^{\leq} \;\; \psi_1^{d_1} \cdots \psi_n^{d_n} \prod_{i=1}^n \frac{(2(d_i-1) + x_i)!!}{(x_i - 2)!!}
\end{equation}
we see that once more the general principle of the conjecture \ref{conj:main} arises: each ratio of factorials is a non-homogeneous polynomial in $e_1$ of degree $n, n-1$ and $n-k$ respectively, and only summands of the form $e_1^a$ or $e_1^ae_k^b$ appears, as expected from a genus one expression. The non-homogeneity is also non surprising, since the class $\Omega^{\leq}$ contains non-trivial contributions in each cohomological degree.


\newpage


\appendix
\section{Table of coefficients $D_{g,n}(\Lambda)$ for $g = 3$ and $n \leq  5$}
\label{sec:Dgnappendix}
We collect below the data of the normalised coefficients $24^g g!C_{g,n}(\lambda)$ for \( g = 3 \), $n \leq 5$, 
and for all \( \lambda \leq 3g-3+n \).

In front of each coefficient we provide the value of
\( 3g-3 - (|\lambda| - \lambda_{1}) \) 
in boldface: whenever this value is negative, we get 
vanishing coefficients (highlighted
in \highlightg{green}), illustrating the vanishing predicted by lemma \ref{lem:1}; whenever \( \ell(\lambda) > g \), we get 
vanishing coefficients (highlighted
in \highlightr{red}), illustrating the vanishing predicted by the main conjecture.

\begin{align*}
\intertext{(g,n) = (3,1)}
	&\textbf{7} \, &D_{3,1}(1, 1, 1, 1, 1, 1, 1) &= 1
	&
& & &
\intertext{(g,n) = (3,2)}
	&\textbf{7} \, &D_{3,2}(1, 1, 1, 1, 1, 1, 1, 1) &= 1
	&
&\textbf{2} \, &D_{3,2}(2, 2, 2, 1, 1) &= -27/7
\\
&\textbf{6} \, &D_{3,2}(2, 1, 1, 1, 1, 1, 1) &= -3
&
&\textbf{0} \, &D_{3,2}\highlightr{(2, 2, 2, 2)} &= 0
\\
&\textbf{4} \, &D_{3,2}(2, 2, 1, 1, 1, 1) &= 27/5
&
\intertext{(g,n) = (3,3)}
	&\textbf{7} \, &D_{3,3}(1, 1, 1, 1, 1, 1, 1, 1, 1) &= 1
	&
&\textbf{2} \, &D_{3,3}(3, 2, 2, 1, 1) &= -81/7
\\
&\textbf{6} \, &D_{3,3}(2, 1, 1, 1, 1, 1, 1, 1) &= -3
&
&\textbf{3} \, &D_{3,3}(3, 3, 1, 1, 1) &= 153/35
\\
&\textbf{6} \, &D_{3,3}(3, 1, 1, 1, 1, 1, 1) &= -39/5
&
&\textbf{0} \, &D_{3,3}\highlightr{(2, 2, 2, 2, 1)} &= 0
\\
&\textbf{4} \, &D_{3,3}(2, 2, 1, 1, 1, 1, 1) &= 27/5
&
&\textbf{1} \, &D_{3,3}(3, 3, 2, 1) &= 27/7
\\
&\textbf{4} \, &D_{3,3}(3, 2, 1, 1, 1, 1) &= 594/35
&
&\textbf{0} \, &D_{3,3}\highlightr{(3, 2, 2, 2)} &= 0
\\
&\textbf{2} \, &D_{3,3}(2, 2, 2, 1, 1, 1) &= -27/7
&
&\textbf{0} \, &D_{3,3}(3, 3, 3) &= -27/35
\intertext{(g,n) = (3,4)}
&\textbf{7} \, &D_{3,4}(1, 1, 1, 1, 1, 1, 1, 1, 1, 1) &= 1
&\hspace{10pt} 
&\textbf{1} \, &D_{3,4}(3, 3, 2, 1, 1) &= 27/7
\\
&\textbf{6} \, &D_{3,4}(2, 1, 1, 1, 1, 1, 1, 1, 1) &= -3
&
&\textbf{2} \, &D_{3,4}(4, 2, 2, 1, 1) &= -54
\\
&\textbf{6} \, &D_{3,4}(3, 1, 1, 1, 1, 1, 1, 1) &= -39/5
&
&\textbf{0} \, &D_{3,4}\highlightr{(3, 2, 2, 2, 1)} &= 0
\\
&\textbf{4} \, &D_{3,4}(2, 2, 1, 1, 1, 1, 1, 1) &= 27/5
&
&\textbf{3} \, &D_{3,4}(4, 3, 1, 1, 1) &= 324/35
\\
&\textbf{4} \, &D_{3,4}(3, 2, 1, 1, 1, 1, 1) &= 594/35
&
&\textbf{0} \, &D_{3,4}\highlightr{(4, 2, 2, 2)} &= 0
\\
&\textbf{6} \, &D_{3,4}(4, 1, 1, 1, 1, 1, 1) &= -594/35
&
&\textbf{2} \, &D_{3,4}(4, 4, 1, 1) &= -1152/35
\\
&\textbf{2} \, &D_{3,4}(2, 2, 2, 1, 1, 1, 1) &= -27/7
&
&\textbf{0} \, &D_{3,4}(3, 3, 3, 1) &= -27/35
\\
&\textbf{2} \, &D_{3,4}(3, 2, 2, 1, 1, 1) &= -81/7
&
&\textbf{-1} \, &D_{3,4}\highlightg{(3, 3, 2, 2)} &= 0
\\
&\textbf{4} \, &D_{3,4}(4, 2, 1, 1, 1, 1) &= 1692/35
&
&\textbf{1} \, &D_{3,4}(4, 3, 2, 1) &= 108/5
\\
&\textbf{3} \, &D_{3,4}(3, 3, 1, 1, 1, 1) &= 153/35
&
&\textbf{0} \, &D_{3,4}(4, 4, 2) &= 144/35
\\
&\textbf{0} \, &D_{3,4}\highlightr{(2, 2, 2, 2, 1, 1)} &= 0
&
&\textbf{0} \, &D_{3,4}(4, 3, 3) &= -162/35
\\
&\textbf{-2} \, &D_{3,4}\highlightg{(2, 2, 2, 2, 2)} &= 0
&
\\
\intertext{(g,n) = (3,5)}
&\textbf{7} \, &D_{3,5}(1, 1, 1, 1, 1, 1, 1, 1, 1, 1, 1) &= 1
&
&\textbf{0} \, &D_{3,5}\highlightr{(4, 2, 2, 2, 1)} &= 0
\\
&\textbf{6} \, &D_{3,5}(2, 1, 1, 1, 1, 1, 1, 1, 1, 1) &= -3
&
&\textbf{-2} \, &D_{3,5}\highlightg{(3, 2, 2, 2, 2)} &= 0
\\
&\textbf{4} \, &D_{3,5}(2, 2, 1, 1, 1, 1, 1, 1, 1) &= 27/5
&
&\textbf{0} \, &D_{3,5}(3, 3, 3, 1, 1) &= -27/35
\\
&\textbf{6} \, &D_{3,5}(3, 1, 1, 1, 1, 1, 1, 1, 1) &= -39/5
&
&\textbf{2} \, &D_{3,5}(4, 4, 1, 1, 1) &= -1152/35
\\
&\textbf{2} \, &D_{3,5}(2, 2, 2, 1, 1, 1, 1, 1) &= -27/7
&
&\textbf{2} \, &D_{3,5}(5, 2, 2, 1, 1) &= -1458/5
\\
&\textbf{6} \, &D_{3,5}(4, 1, 1, 1, 1, 1, 1, 1) &= -594/35
&
&\textbf{3} \, &D_{3,5}(5, 3, 1, 1, 1) &= 2844/35
\\
&\textbf{4} \, &D_{3,5}(3, 2, 1, 1, 1, 1, 1, 1) &= 594/35
&
&\textbf{1} \, &D_{3,5}(4, 3, 2, 1, 1) &= 108/5
\\
&\textbf{4} \, &D_{3,5}(4, 2, 1, 1, 1, 1, 1) &= 1692/35
&
&\textbf{-1} \, &D_{3,5}\highlightg{(4, 3, 2, 2)} &= 0
\\
&\textbf{2} \, &D_{3,5}(3, 2, 2, 1, 1, 1, 1) &= -81/7
&
&\textbf{0} \, &D_{3,5}(4, 4, 2, 1) &= 144/35
\\
&\textbf{3} \, &D_{3,5}(3, 3, 1, 1, 1, 1, 1) &= 153/35
&
&\textbf{0} \, &D_{3,5}\highlightr{(5, 2, 2, 2)} &= 0
\\
&\textbf{0} \, &D_{3,5}\highlightr{(2, 2, 2, 2, 1, 1, 1)} &= 0
&
&\textbf{0} \, &D_{3,5}(4, 3, 3, 1) &= -162/35
\\
&\textbf{6} \, &D_{3,5}(5, 1, 1, 1, 1, 1, 1) &= -2286/35
&
&\textbf{1} \, &D_{3,5}(5, 3, 2, 1) &= 4572/35
\\
&\textbf{-2} \, &D_{3,5}\highlightg{(2, 2, 2, 2, 2, 1)} &= 0
&
&\textbf{2} \, &D_{3,5}(5, 4, 1, 1) &= -5904/35
\\
&\textbf{2} \, &D_{3,5}(4, 2, 2, 1, 1, 1) &= -54
&
&\textbf{-2} \, &D_{3,5}\highlightg{(3, 3, 3, 2)} &= 0
\\
&\textbf{3} \, &D_{3,5}(4, 3, 1, 1, 1, 1) &= 324/35
&
&\textbf{1} \, &D_{3,5}(5, 5, 1) &= 432/5
\\
&\textbf{0} \, &D_{3,5}\highlightr{(3, 2, 2, 2, 1, 1)} &= 0
&
&\textbf{0} \, &D_{3,5}(5, 4, 2) &= 144/5
\\
&\textbf{4} \, &D_{3,5}(5, 2, 1, 1, 1, 1) &= 8532/35
&
&\textbf{0} \, &D_{3,5}(5, 3, 3) &= -162/5
\\
&\textbf{1} \, &D_{3,5}(3, 3, 2, 1, 1, 1) &= 27/7
&
&\textbf{-1} \, &D_{3,5}\highlightg{(4, 4, 3)} &= 0
\\
&\textbf{-1} \, &D_{3,5}\highlightg{(3, 3, 2, 2, 1)} &= 0
&
\end{align*}
\newpage

\begin{landscape}
\section{Table of all coefficients $C_g(\lambda)$ for $g =4$}
\label{sec:Cgappendix}
Let us apply corollary \ref{lem:poly} to compute \( 24^g g! C_{g}(k, \mu) \) for $g=4$, for all \( \mu \in \mathcal{Q}_{g} \) and for all \( k \) . We moreover verify numerically the lower bound \( k_{0} = 3g-2-|\mu| + \mu_{1} - \delta_{|\mu|, 3g-3} \) for the polynomial behaviour of $C_4(k,\mu)$ in $k$. There are $30$ partitions to consider:
	\begin{scriptsize}
\begin{align*}
	\mathcal{Q}_4 = 
	\{
		&\varnothing, (2), (3), (4), (5), (6), (7), (8), (9), (2, 2), (3,2),\\
		&(4, 2), (3, 3), (5, 2), (4, 3), (6, 2), (5, 3), (4, 4), (7,2),\\
		&(6, 3), (5, 4), (2, 2, 2), (3, 2, 2), (4, 2, 2), (3, 3, 2),\\
		&(5, 2, 2), (4, 3, 2), (3, 3, 3), (2, 2, 2, 2), (3, 2, 2, 2)
\}
\end{align*}
\vspace{10pt}
	\end{scriptsize}
	Start with partitions of length \( \ell(\mu) \leq 1 \), here \( k_{0} = 10 \) 
	except for the max size partition \( (9) \) where \( k_{0} = 9 \).
	\begin{scriptsize}
\begin{align*}	
	C_4(\emptyset) &= 1	 	& C_4(2,(2))  &= 54/5 		& C_4(3,(3)) &= 4842/175 	&C_4(4,(4)) &= -27648/175 	& C_4(5,(5)) &= -9792/35 \\
C_4(2,\varnothing) &= -4                & C_4(3,(2))  &= 324/7          & C_4(4,(3)) &= 14184/175       &C_4(5,(4)) &= -182592/175      & C_4(6,(5)) &= -51264/175 \\
C_4(3,\varnothing) &= -68/5             & C_4(4,(2))  &= 21816/175      & C_4(5,(3)) &= 105768/175      &C_4(6,(4)) &= -36288/5         & C_4(7,(5)) &= -269568/25 \\
C_4(4,\varnothing) &= -1144/35          & C_4(5,(2))  &= 17064/25       & C_4(6,(3)) &= 104112/25       &C_4(7,(4)) &= -11405952/175    & C_4(8,(5)) &= -10245888/175 \\
C_4(5,\varnothing) &= -21816/175        & C_4(6,(2))  &= 117648/25      & C_4(7,(3)) &= 227952/7        &C_4(8,(4)) &= -22470912/35     & C_4(9,(5)) &= -47778048/175 \\
C_4(6,\varnothing) &= -141264/175       & C_4(7,(2))  &= 266832/7       & C_4(8,(3)) &= 51469056/175    &C_4(9,(4)) &= -176332032/25 \\
C_4(7,\varnothing) &= -1106064/175      & C_4(8,(2))  &= 12432384/35    & C_4(9,(3)) &= 521255808/175 \\
C_4(8,\varnothing) &= -9988992/175      & C_4(9,(2))  &= 93011328/25 \\
C_4(9,\varnothing) &= -102117888/175
\end{align*}
\end{scriptsize}
\\[10pt]
\begin{scriptsize}
\begin{gather*}
C_4(k,\varnothing) = - \frac{ \left(k - 5\right)! }{2270268000}
\big(3717 k^{9} + 345264 k^{8} + 10022652 k^{7} + 51871810 k^{6} - 1143710229 k^{5} + 8585898070 k^{4} 
- 52358293308 k^{3} + 196752773416 k^{2} - 387514181568 k + 314153693568\big)
\\
C_4(k,(2))= \frac{\left(k - 3\right)!}{756756000}
\left(11151 k^{7} + 998761 k^{6} + 29709011 k^{5} + 269630259 k^{4} - 712790674 k^{3} + 5966416652 k^{2} - 20766829960 k + 26179474464\right) 
\\
C_4(k,(3)) = \frac{\left(k - 2\right)!}{756756000}
\left(723 k^{6} + 199787 k^{5} + 16387577 k^{4} + 413083225 k^{3} - 51459908 k^{2} + 3514829748 k - 1065697776\right) 
\\
C_4(k,(4)) = - \frac{\left(k - 1\right)!}{378378000}
\left(25569 k^{5} + 1771714 k^{4} + 37277747 k^{3} + 274984058 k^{2} + 184540456 k + 1946520576\right) 
\\
C_4(k,(5)) = \frac{k!}{378378000}
\left(12093 k^{4} + 930010 k^{3} + 14966271 k^{2} - 218527846 k - 287488200\right) 
\end{gather*}
\end{scriptsize}
\\[10pt]
\begin{scriptsize}
\begin{align*}
C_4(6,(6)) &= 1787904/175 	& C_4(7,(7)) &= -13825152/175 		& C_4(8,(8)) &= 2078208/7 \\
C_4(7,(6)) &= 14375808/175      & C_4(8,(7)) &= -28266624/35            & C_4(9,(8)) &= 24938496/7 \\
C_4(8,(6)) &= 157378176/175     & C_4(9,(7)) &= -339110784/35 \\
C_4(9,(6)) &= 49168512/5\\
C_4(k,(6)) & = \frac{\left(k + 1\right)!}{15765750}
\left(513 k^{3} + 78137 k^{2} + 2328884 k + 15061740\right) 
	   &
C_4(k,(7)) &= - \frac{\left(k + 2\right)!}{7882875}
\left(2011 k^{2} + 66173 k + 1154393\right) 
\end{align*}
\end{scriptsize}
\\[10pt]
\begin{scriptsize}
	\begin{gather*}
\\
\\
C_4(k,(8)) = \frac{\left(k + 3\right)!}{1433250} \left(99 k + 9868\right) 
\\
C_4(k,(9)) = - \frac{11 \left(k + 4\right)!}{159250}
	\end{gather*}
\end{scriptsize}
\\[10pt]
\begin{scriptsize}
\begin{align*}
C_4(2,(2, 2)) &= -108/7 	& C_4(3,(3, 2)) &= -180/7 	& C_4(4,(4, 2)) &= 2880/7 	& C_4(3,(3, 3)) &= -36/5 	& C_4(5,(5, 2)) &= -242496/175 \\
C_4(3,(2, 2)) &= -468/7         & C_4(4,(3, 2)) &= -1584/35     & C_4(5,(4, 2)) &= 91008/35     & C_4(4,(3, 3)) &= -216/5       & C_4(6,(5, 2)) &= -409536/35 \\
C_4(4,(2, 2)) &= -1800/7        & C_4(5,(3, 2)) &= -24624/35    & C_4(6,(4, 2)) &= 3836736/175  & C_4(5,(3, 3)) &= -8712/35     & C_4(7,(5, 2)) &= -19225728/175 \\
C_4(5,(2, 2)) &= -57096/35      & C_4(6,(3, 2)) &= -31392/7     & C_4(7,(4, 2)) &= 34399872/175 & C_4(6,(3, 3)) &= -51408/25 \\
C_4(6,(2, 2)) &= -410544/35     & C_4(7,(3, 2)) &= -5870304/175 \\
C_4(7,(2, 2)) &= -3450672/35 \\
\end{align*}
\end{scriptsize}
\\[10pt]
\begin{scriptsize}
	\begin{gather*}
C_4(k,(2, 2)) = - \frac{\left(1515 k^{5} + 125102 k^{4} + 3592493 k^{3} + 38382106 k^{2} + 33387016 k + 313049088\right) \left(k - 1\right)!}{29106000}
\\
C_4(k,(3, 2)) = \frac{\left(450 k^{4} + 24362 k^{3} + 67239 k^{2} - 7592459 k - 8017332\right) k!}{7276500}
\\
C_4(k,(4, 2)) = \frac{\left(371 k^{3} + 20182 k^{2} + 325180 k + 2519512\right) \left(k + 1\right)!}{1212750}
\\
C_4(k,(3, 3)) = - \frac{\left(69 k^{3} + 3077 k^{2} - 14 k + 233736\right) \left(k + 1\right)!}{882000}
\\
C_4(k,(5, 2)) = - \frac{\left(389 k^{2} + 29593 k + 509174\right) \left(k + 2\right)!}{2425500}
	\end{gather*}
\end{scriptsize}
\\[10pt]
\begin{scriptsize}
\begin{align*}
C_4(4,(4, 3)) &= -1152/35 	& C_4(6,(6, 2)) &= -10368/25 & C_4(5,(5, 3)) &= 119232/175 & C_4(4,(4, 4)) &= -6912/175 \\
C_4(5,(4, 3)) &= -22464/175      & C_4(7,(6, 2)) &= -13824/7 & C_4(6,(5, 3)) &= 211392/35 & C_4(5,(4, 4)) &= -55296/175 \\
C_4(6,(4, 3)) &= -333504/175 & \\
C_4(k,(4, 3)) &= 
- \frac{\left(27 k^{2} + 671 k + 5292\right) \left(k + 2\right)!}{220500}
	      &
C_4(k,(6, 2)) &= \frac{\left(191 k - 2532\right) \left(k + 3\right)!}{1212750}
	      &
C_4(k,(5, 3)) &= \frac{\left(38 k + 1673\right) \left(k + 3\right)!}{110250}
	      &
C_4(k,(4, 4)) &= - \frac{2 \left(7 k + 188\right) \left(k + 3\right)!}{55125}
\end{align*}
\end{scriptsize}
\\[10pt]
\begin{scriptsize}
\begin{align*}
C_4(k,(7, 2)) &= \frac{107 \left(k + 4\right)!}{242550}
	      &
C_4(k,(6, 3)) &= - \frac{11 \left(k + 4\right)!}{18375}
	      &
C_4(k,(5, 4)) &= \frac{2 \left(k + 4\right)!}{7875}
\end{align*}
\end{scriptsize}
\\[10pt]
\begin{scriptsize}
\begin{align*}
C_4(2,(2, 2, 2)) &= 9 	&C_4(3,(3, 2, 2)) &= -18		& C_4(4,(4, 2, 2)) &= -1728/35 & C_4(3,(3, 3, 2)) &= 36/5 \\
C_4(3,(2, 2, 2)) &= 36  &C_4(4,(3, 2, 2)) &= -648/5 		& C_4(5,(4, 2, 2)) &= -2880/7 & C_4(4,(3, 3, 2)) &= 1944/35 \\
C_4(4,(2, 2, 2)) &= 216 &C_4(5,(3, 2, 2)) &= -32184/35 \\
C_4(5,(2, 2, 2)) &= 7128/5 \\
\end{align*}
\end{scriptsize}
\\[10pt]
\begin{scriptsize}
\begin{gather*}
C_4(k,(2, 2, 2)) = \frac{\left(35 k^{3} + 2499 k^{2} + 65830 k + 651864\right) \left(k + 1\right)!}{529200}
\\
C_4(k,(4, 2, 2)) = - \frac{\left(7 k + 188\right) \left(k + 3\right)!}{22050}
\\
C_4(k,(3, 2, 2)) = - \frac{\left(35 k^{2} + 1701 k + 22788\right) \left(k + 2\right)!}{176400}
\\
C_4(k,(3, 3, 2)) = \frac{\left(7 k + 180\right) \left(k + 3\right)!}{19600}
\end{gather*}
\end{scriptsize} 
\\[10pt]
\begin{scriptsize}
\begin{align}
C_4(k,(5, 2, 2)) &= - \frac{\left(k + 4\right)!}{11025}
		 &
C_4(k,(4, 3, 2)) &= \frac{\left(k + 4\right)!}{2450}
		 &
C_4(k,(3, 3, 3)) &= - \frac{\left(k + 4\right)!}{3920}
		 &
		 \end{align}
	 \end{scriptsize}
	 \\[10pt]
	 \begin{scriptsize}
		 \begin{align}
C_4(k,(2, 2, 2, 2)) &= 0
		    &
C_4(k,(3, 2, 2, 2)) &= 0
\end{align}
\end{scriptsize}
Therefore we again verify the conjecture for \( g = 4 \),
with \( C_{g}((k,2,2,2,2) = 0 \) for all \( k \geq 2 \)  
and \( C_{4}(k,3,2,2,2) = 0 \) 
for all \( k \geq 3\).
\end{landscape}
\newpage


%

\printbibliography

\end{document}